\documentclass{article}

\usepackage{alltt}
\usepackage{amsfonts}
\usepackage{amsmath}
\usepackage{amssymb}
\usepackage{amsthm}
\usepackage{booktabs}
\usepackage{caption}
\usepackage{subcaption}
\usepackage{enumitem}
\usepackage{fancyhdr}
\usepackage{fullpage}
\usepackage{graphicx}
\usepackage{mathdots}
\usepackage{mathtools}
\usepackage{microtype}
\usepackage{multirow}
\usepackage{pdflscape}
\usepackage{pgfplots}
\usepackage{siunitx}
\usepackage{slashed}
\usepackage{tabularx}
\usepackage{tikz}
\usepackage{tkz-euclide}
\usepackage[normalem]{ulem}
\usepackage[all]{xy}
\usepackage{imakeidx}
\usepackage{verbatim}
\usepackage[
backend=bibtex,
style=numeric,
]{biblatex}
\usepackage{url}
\usepackage{breakurl}
\usepackage[breaklinks]{hyperref}

\usetikzlibrary{arrows}
\usetikzlibrary{arrows.meta}
\usetikzlibrary{decorations.markings}
\usetikzlibrary{decorations.pathmorphing}
\usetikzlibrary{positioning}
\usetikzlibrary{fadings}
\usetikzlibrary{intersections}
\usetikzlibrary{cd}

\pgfarrowsdeclarecombine{twolatex'}{twolatex'}{latex'}{latex'}{latex'}{latex'}
\tikzset{->/.style = {decoration={markings,
			mark=at position 1 with {\arrow[scale=2]{latex'}}},
		postaction={decorate}}}
\tikzset{<-/.style = {decoration={markings,
			mark=at position 0 with {\arrowreversed[scale=2]{latex'}}},
		postaction={decorate}}}
\tikzset{<->/.style = {decoration={markings,
			mark=at position 0 with {\arrowreversed[scale=2]{latex'}},
			mark=at position 1 with {\arrow[scale=2]{latex'}}},
		postaction={decorate}}}
\tikzset{->-/.style = {decoration={markings,
			mark=at position #1 with {\arrow[scale=2]{latex'}}},
		postaction={decorate}}}
\tikzset{-<-/.style = {decoration={markings,
			mark=at position #1 with {\arrowreversed[scale=2]{latex'}}},
		postaction={decorate}}}
\tikzset{->>/.style = {decoration={markings,
			mark=at position 1 with {\arrow[scale=2]{latex'}}},
		postaction={decorate}}}
\tikzset{<<-/.style = {decoration={markings,
			mark=at position 0 with {\arrowreversed[scale=2]{twolatex'}}},
		postaction={decorate}}}
\tikzset{<<->>/.style = {decoration={markings,
			mark=at position 0 with {\arrowreversed[scale=2]{twolatex'}},
			mark=at position 1 with {\arrow[scale=2]{twolatex'}}},
		postaction={decorate}}}
\tikzset{->>-/.style = {decoration={markings,
			mark=at position #1 with {\arrow[scale=2]{twolatex'}}},
		postaction={decorate}}}
\tikzset{-<<-/.style = {decoration={markings,
			mark=at position #1 with {\arrowreversed[scale=2]{twolatex'}}},
		postaction={decorate}}}

\tikzset{circ/.style = {fill, circle, inner sep = 0, minimum size = 3}}
\tikzset{scirc/.style = {fill, circle, inner sep = 0, minimum size = 1.5}}
\tikzset{mstate/.style={circle, draw, blue, text=black, minimum width=0.7cm}}

\tikzset{eqpic/.style={baseline={([yshift=-.5ex]current bounding box.center)}}}
\tikzset{commutative diagrams/.cd,cdmap/.style={/tikz/column 1/.append style={anchor=base east},/tikz/column 2/.append style={anchor=base west},row sep=tiny}}

\theoremstyle{definition}

\newtheorem*{notation}{Notation}

\newtheorem{nthm}{Theorem}[section]

\newtheorem{nlemma}[nthm]{Lemma}
\newtheorem{nprop}[nthm]{Proposition}
\newtheorem{assumption}[nthm]{Assumption}
\newtheorem{assumptions}[nthm]{Assumptions}

\newtheorem{ncon}[nthm]{Conjecture}
\newtheorem{defi}[nthm]{Definition}
\newtheorem{eg}[nthm]{Example}

\newtheorem{remark}[nthm]{Remark}

\newcommand{\al}{\alpha}
\newcommand{\ga}{\gamma}

\newcommand{\OC}{\mathcal{OC}}
\newcommand{\MC}{\mathcal{C}}
\newcommand{\MD}{\mathcal{D}}
\newcommand{\MM}{\mathcal{M}}
\newcommand{\MK}{\mathcal{K}}

\newcommand{\E}{\mathcal{E}}
\newcommand{\V}{\mathcal{V}}
\newcommand{\W}{\mathcal{W}}
\newcommand{\F}{\mathcal{F}}

\newcommand{\qb}{\mathfrak{q}^{b,\ga}}
\newcommand{\mb}{\mathfrak{m}^{b,\ga}}
\newcommand{\m}{\mathfrak{m}}

\newcommand{\q}{\mathfrak{q}}
\newcommand{\CC}{\mathbb{C}}
\newcommand{\KK}{\mathbb{K}}
\newcommand{\QQ}{\mathbb{Q}}
\newcommand{\RR}{\mathbb{R}}
\newcommand{\ZZ}{\mathbb{Z}}

\newcommand{\WE}{\widetilde{\E}}

\newcommand{\DX}{QH^*(X)[[u]]}
\newcommand{\YY}{Y}

\newcommand{\WV}{\widetilde{\V}}
\newcommand{\WW}{\widetilde{\W}}
\newcommand{\WR}{\widetilde{R}}
\newcommand{\bfe}{\mathbf{e}}

\let\stdsection\section
\renewcommand\section{\newpage\stdsection}

\def\st{\bgroup \ULdepth=-.55ex \ULset}
\begin{document}
	
	\title{Open Gromov-Witten invariants from the Fukaya category}
	
	\date{\today}
	
	\author{Kai Hugtenburg}
	
	\maketitle
	\begin{abstract}
		This paper proposes a framework to show that the Fukaya category of a symplectic manifold $X$ determines the open Gromov-Witten invariants of Lagrangians $L \subset X$. We associate to an object in an $A_\infty$-category an extension of the negative cyclic homology, called \emph{relative cyclic homology}. We extend the Getzler-Gauss-Manin connection to relative cyclic homology. Then, we construct (under simplifying technical assumptions) a relative cyclic open-closed map, which maps the relative cyclic homology of a Lagrangian $L$ in the Fukaya category of a symplectic manifold $X$ to the $S^1$-equivariant relative quantum homology of $(X,L)$. Relative quantum homology is the dual to the relative quantum cohomology constructed by Solomon-Tukachinsky. This is an extension of quantum cohomology, and comes equipped with a connection extending the quantum connection. We prove that the relative open-closed map respects connections. As an application of this framework, we show, assuming a construction of the relative cyclic open-closed map in a broader technical setup, that the Fukaya category of a Calabi-Yau variety determines the open Gromov-Witten invariants with one interior marked point for any null-homologous Lagrangian brane. This in particular includes the open Gromov-Witten invariants of the real locus of the quintic threefold considered in \cite{PSW}. 
	\end{abstract}
\tableofcontents
\section{Introduction}
Classical enumerative mirror symmetry for the quintic threefold states (see \cite{Can}) that the closed Gromov-Witten invariants of the quintic threefold, $X$, can be computed from the period integrals of a mirror family $Y$. Morrison \cite{Mor} rephrased enumerative mirror symmetry as an isomorphism of variations of Hodge structures (VHS) associated to $X$ and $Y$. On $Y$ the VHS is given by the hodge structure on $H^*(Y)$, the connection is the Gauss-Manin connection. The VHS associated to $X$ is given by quantum cohomology, equipped with the quantum connection. Morrison furthermore showed how a splitting of the Hodge filtration of a VHS, arising from the monodromy weight filtration, gives rise to enumerative invariants. For $X$, these are the genus $0$ Gromov-Witten invariants.

In his ICM address, Kontsevich conjectured a homological mirror symmetry, which in the Calabi-Yau setting states that for a symplectic space $X$ there exists a mirror (complex) space $Y$ such that there is a quasi-equivalence: \begin{equation}
	D^\pi Fuk(X) \cong D^b Coh(Y).
\end{equation}
Kontsevich moreover conjectured that this homological mirror symmetry implies enumerative mirror symmetry.

For Calabi-Yau hypersurfaces in projective space, Sheridan \cite{She12} proves homological mirror symmetry. Ganatra-Perutz-Sheridan \cite{GPS} show that for certain Calabi-Yaus, whose VHS is of Hodge-Tate type, the Gromov-Witten invariants predicted in \cite{Can} are indeed extractable from the Fukaya category. The main tool here is the negative cyclic open closed map $\OC^-$, constructed by \cite{FOOO} and \cite{Ga19}, which maps the VHS associated to the Fukaya category, $HC_*^-(Fuk(X))$, to the quantum cohomology. The work by Ganatra-Perutz-Sheridan relies on work they announced, \cite{GPS2}, which shows that $\OC^-$ indeed respects the structure of the VHS.

Homological mirror symmetry is understood to apply in more general settings, such as when $X$ is Fano. In this case the mirror $Y$ is a Landau-Ginzburg model and comes equipped with a superpotential $W: Y \rightarrow \CC$. Barannikov \cite{Bar} defines the appropriate analogue of a VHS in the Fano setting: a variation of semi-infinite Hodge structures (VSHS), also known as a formal TP-structure. To homogenise notation, we include the definition:
\begin{defi}[{Formal TEP-structure, see \cite{Her}}]\text{ }
	\label{formal TEP structures defi}
	Let $\mathbb{K}$ be a field of characteristic 0. Let $R$ be a $\ZZ/2$-graded commutative $\KK$-algebra. Let $\mathcal{E}$ be a free and finitely-generated $\ZZ/2$-graded $R[[u]]$-module, with $u$ of even degree.
	\begin{enumerate}
		\item A formal T-structure is a connection $\nabla: Der_{\KK}R \otimes \mathcal{E} \rightarrow u^{-1} \mathcal{E}$, which is flat and of even degree.
		\item A formal TE-structure is a formal T-structure together with an extension of the connection to a flat connection $\nabla: Der_{\KK}(R[[u]]) \otimes \mathcal{E} \rightarrow u^{-2} \mathcal{E}$.
		\item A formal TEP-structure is a formal T-structure equipped with a polarisation, i.e.\ a covariantly constant pairing $$( \cdot, \cdot )_{\E}: \mathcal{E} \otimes \mathcal{E} \rightarrow R[[u]],$$ which is $R$-linear, of degree $0$ and $u$-sesquilinear. That is, for $f(u) \in R[[u]]$, we have: $$f(u)(a,b)_\E = (f(u)a,b)_\E = (-1)^{n+|f||a|}(a,f(-u)b)_\E.$$ Here $n \in \ZZ/2$ is called the \emph{dimension} of $\E$. We moreover require the restriction of the pairing $( \cdot, \cdot )_{\WE}: \WE \otimes \WE \rightarrow R$ to be non-degenerate. Here $\WE = \E/u\E$.
	\end{enumerate}
A VSHS is the same as a formal TP-structure, and a VHS is a $\ZZ$-graded VSHS.
\end{defi}
\begin{remark}
	Hertling \cite{Her} defines (non-formal) TERP-structures. In that case, one works with convergent power series in $u$, instead of the formal power series we consider. We will always only be working with formal TEP-structures, we will thus drop the label `formal'. Additionally, Hertling considers a real structure (the R-structure), which we will not talk about.
\end{remark}
\begin{eg}
	The quantum TEP-structure is defined over $R = \Lambda$, where $\Lambda = \CC((Q^{\RR}))$ is the Novikov field. It is given by the $S^1$-equivariant quantum cohomology $QH^*(X;R)[[u]]$. The connection is as defined in \cite{Dub}. The pairing is given by the sesquilinear extension of the Poincar\'e pairing. One can additionally incorporate bulk-deformations to obtain a TEP-structure over a larger base $R = \Lambda[[H^*(X)]]$.
\end{eg}
\begin{eg}
	Let $\MC$ be an $R$-linear Calabi-Yau $A_\infty$-category. Getzler constructs a flat connection \begin{equation}\nabla^{GGM}: Der_{\KK}R \otimes HC_*^-(\mathcal{C}) \rightarrow u^{-1} HC_*^-(\mathcal{C}).\end{equation} Shklyarov \cite{Shk} and Katzarkov-Kontsevich-Pantev \cite{KKP} extend this connection in the $u$-direction. Shklyarov also defines a pairing on $HC_*^-(\mathcal{C})$, called the higher residue pairing. This pairing was shown to be covariantly constant by Shklyarov for the $u$-connection, and for the Getzler-Gauss-Manin connection by Sheridan \cite{She}. In general, $HC^-_*(\MC)$ might not be free and finitely-generated. This is conjectured to hold for smooth and compact $A_\infty$-categories, see \cite{KS1}. In the $\ZZ$-graded setting, Kaledin \cite{Ka} proves this conjectures holds. We will always assume this conjecture holds, so that $HC^-_*(\MC)$ is endowed with a TEP-structure.
\end{eg}
\begin{ncon}
	\label{VSHS conjecture}
	There exists a cyclic open-map $\mathcal{OC}^{-}: HC^-_*(Fuk(X)) \rightarrow QH^*(X;\Lambda)[[u]]$. This is a morphism of TEP-structures over $\Lambda$. By including bulk-deformation, one can extend this to a morphism of TEP-structures over $R = \Lambda[[H^*(X)]]$.
\end{ncon}
Such a morphism has not been constructed in general. Partial results exist: \cite{FOOO} and \cite{Ga19} construct cyclic open-closed maps in a wide range of settings. Ganatra-Perutz-Sheridan \cite{GPS} have announced work proving this is an isomorphism of TP-structures when $X$ is a projective Calabi-Yau manifold. In \cite{Hug} the author proves (under the same technical assumptions as in this paper, see section \ref{assumptions}) a local version of this conjecture, constructing a morphism of TE-structures $\OC^-: HC^-_*(CF^*(L,L)) \rightarrow QH^*(X;R)[[u]]$. There, all bulk-deformations by elements in $H^*(X)$ are taken into account.

\begin{defi}
	A \emph{splitting} of an TEP-structure is a $\KK$-linear map $s: \WE \rightarrow \E$ which splits the canonical map $\pi: \E \rightarrow \WE$, such that $(s(a),s(b))_\E = (a,b)_{\WE}$ for all $a,b \in \WE$.
\end{defi}
Barannikov shows that a TEP-structure together with a splitting gives rise to enumerative invariants. For Gromov-Witten invariants, one takes the splitting:
\begin{align}
	s^{GW}: QH^*(X) &\rightarrow \DX\nonumber\\
	\alpha &\mapsto \alpha.
\end{align}

To obtain enumerative invariant from the Fukaya category, one thus needs to characterise the Gromov-Witten splitting on $HC^-_*(Fuk(X))$. Such a characterisation has not been obtained in general. When $X$ is Calabi-Yau and $HC^-_*(Fuk(X))$ is of Hodge-Tate type, Ganatra-Perutz-Sheridan \cite{GPS} characterise the Gromov-Witten splitting. Amorim-Tu \cite{AT} do this when $HH^*(Fuk(X))$ is a semi-simple ring.

Of additional interest in the Fano setting is the E-structure (the $u$-connection) on $HC^-_*(Fuk(X))$ and on quantum cohomology. This was considered by the author in \cite{Hug}, where it was shown that the negative cyclic open-closed map respects the E-structure.
 
We now return to the quintic threefold. In \cite{Wal} and \cite{WM} Morrison and Walcher propose an extension of the classical picture of enumerative mirror symmetry by including open Gromov-Witten invariants of the real locus of the Fermat quintic, which were defined in \cite{So}, and the normal function associated to a family of curves in $Y$ (defined by \cite{Grif}). This extended enumerative mirror symmetry can be encapsulated by an isomorphism of \emph{extensions of VHS}. Walcher uses this to predict values for all genus $0$ open Gromov-Witten invariants of the real locus of the quintic. These predictions were later verified in \cite{PSW}. This paper proposes a framework to show that homological mirror symmetry implies this extended enumerative mirror symmetry.

The main inspiration is the construction by Solomon-Tukachinsky \cite{ST2} of the relative quantum T-structure, $QH^*(X,L)[[u]]$. This is a modification of quantum cohomology which additionally incorporates open Gromov-Witten invariants. We additionally define a connection in the $u$-direction, and show this equips $QH^*(X,L)[[u]]$ with a TE-structure. When the Lagrangian is null-homologous, i.e $[L] = 0 \in H_n(X;\QQ)$, this yields an extension of TE-structures:\begin{equation}
	0 \rightarrow \Lambda[[u]] \rightarrow QH^*(X,L;\Lambda)[[u]] \rightarrow QH^*(X;\Lambda)[[u]] \rightarrow 0.
\end{equation}
The TE-structure on $\Lambda[[u]]$ is defined in Theorem \ref{thm: rel OC respects TE introduction}. To compare this with categorical notions, we need their duals. To this end, let $QH_*(X,L;\Lambda)[[u]]$ be the dual TE-structure, called \emph{relative quantum homology}. Similarly we write $QH_*(X;\Lambda)[[u]]$ for the dual of $QH^*(X;\Lambda)[[u]]$.

On the categorical side, given a unital $A_\infty$-category $\MC$ over a ring $R$ and an object $L \in \MC$, we construct a version of negative cyclic homology called \emph{relative cyclic homology}, denoted $HC^-_*(\MC,L)$. We equip this with a TE-structure (again assuming the non-commutative Hodge-deRham degeneration conjecture holds). If $L$ has zero Chern character $ch(L) = [\mathbf{e}_L] = 0 \in HC^-_*(\MC)$, it fits into an exact sequence of TE-structures: \begin{equation}
	0 \rightarrow HC^-_*(\MC) \rightarrow HC^-_*(\MC,L) \rightarrow R[[u]] \rightarrow 0.
\end{equation}
\begin{remark}
	In Section \ref{section: relative cyclic homology} we construct a more general version of relative cyclic homology associated to any $A_\infty$-functor $F: \MC \rightarrow \MD$. The above construction is obtained by taking $F$ to be the inclusion of the object $L$ into $\MC$.
\end{remark} 
We conjecture: \begin{ncon}
\label{con: relative cyclic OC respects connections}
There exists a relative cyclic open-closed map \begin{equation*}
	\OC^-_L: HC^-_*(Fuk(X),L) \rightarrow QH_*(X,L)[[u]],
\end{equation*}which is a morphism of TE-structures. When $L$ is null-homologous it fits into the commutative diagram:
\[
\begin{tikzcd}
	0 \ar[r] &HC_*^-(Fuk(X)) \ar[d, "\mathcal{OC}^{-}"] \ar[r] & HC_*^-(Fuk(X),L) \ar[d, "\mathcal{OC}_L^{-}"] \ar[r] & \Lambda[[u]]  \ar[d,equal] \ar[r] &0\\
	0 \ar [r] &QH_*(X;\Lambda)[[u]] \ar[r] & QH_*(X,L;\Lambda)[[u]] \ar[r] & \Lambda[[u]] \ar[r] &0
\end{tikzcd}
\]
\end{ncon}
We use the same technical setup as Solomon-Tukachinksy \cite{ST3}, where a Fukaya $A_\infty$-algebra $CF^*(L,L)$ is constructed using differential forms. We then prove the above conjecture in this local setting:
\begin{nthm}[{Theorem \ref{thm: relative OC is a morphism of TE-structures}}]
	\label{thm: rel OC respects TE introduction}
	Let $L\subset X$ be an, oriented, relatively-spin Lagrangian submanifold equipped with a $U(\Lambda)$-local system. Suppose there exists a complex structure $J$ such that $(L,J)$ satisfy Assumptions \ref{assumptions}. Furthermore, suppose $L$ is equipped with a bounding pair $(b,\ga)$ with curvature $c \in R$ (see section \ref{section: bulk-deformations and bounding cochains}). Then there exists a relative cyclic open-closed map $$\OC^-_L: HC^-_*(CF^*(L,L),L) \rightarrow QH_*(X,L;R)[[u]],$$ which is a morphism of TE-structures over $R$. When $L$ is null-homologous, it fits into the commutative diagram:
	\[
	\begin{tikzcd}
		0 \ar[r] &HC_*^-(CF^*(L,L)) \ar[d, "\mathcal{OC}^{-}"] \ar[r] & HC_*^-(CF^*(L,L),L) \ar[d, "\mathcal{OC}_L^{-}"] \ar[r] & R[[u]]  \ar[d,equal] \ar[r] &0\\
		0 \ar [r] &QH_*(X;R)[[u]] \ar[r] & QH_*(X,L;R)[[u]] \ar[r] & R[[u]] \ar[r] &0
	\end{tikzcd}
	\]
	Here $R[[u]]$ is the TE-structure equipped with the connection: \begin{align*}
		\nabla_v(f) &= v(f) - v(c)f\\
		\nabla_{\partial_u}(f) &= \partial_u(f) - \frac{f}{2u} + \frac{cf}{u^2}
	\end{align*}
\end{nthm}
Similar to how $HC^-_*(Fuk(X))$ contains information about closed Gromov-Witten invariants, we propose $HC^-_*(Fuk(X),L)$ can be used to extract open Gromov-Witten invariants of $L$ from the Fukaya category. In general, it is not clear how the extension of TE-structures determines the open Gromov-Witten invariants. However, as evidence for this proposal, we prove:
 \begin{nthm}
	\label{thm: OGW from Fuk quintic threefold}
	Let $X$ be any Calabi-Yau and $L \subset X$ any Lagrangian brane with $[L] = 0 \in H_n(X)$. Then, assuming conjecture \ref{con: relative cyclic OC respects connections} has been proven for $(X,L)$, the Fukaya category of $X$ determines the total open Gromov-Witten invariant: \begin{equation} 
		OGW_{\star,0}(\_) = \sum_{d \in H_2(X,L)} q^{\omega(d)} OGW_{d,0}(\_): H^{n-1}(X) \rightarrow \Lambda.
	\end{equation}
	Here $OGW_{d,0}: H^{n-1}(X) \rightarrow \RR$ denote the genus $0$ open Gromov-Witten invariants in degree $d \in H_2(X,L)$ with no boundary marked points, as constructed in \cite{ST2}.
\end{nthm}
\begin{remark}
	In \cite{ST2} open Gromov-Witten invariants of a null-homologous Lagrangian depend on the choice of cycle $C$ such that $\partial C = L$. In Lemma \ref{lem: unique bounding cycle} we show that for null-homologous Lagrangians in a K\"ahler manifold, there always exists a canonical choice of bounding cycle $C$. Thus we can talk about \emph{the} open Gromov-Witten invariants of such $L$.
\end{remark}
\begin{remark}
	In the situation of Theorem \ref{thm: OGW from Fuk quintic threefold}, the degree axiom shows that the only non-trivial open Gromov-Witten invariants with no boundary marked points and one interior marked point, are those with interior constraints on $H^{n-1}(X)$.
\end{remark}

In the Fano setting we can say more about the extension. For Fanos, we can set the Novikov parameter equal to $1$, and we additionally set all bulk parameters equal to $0$, so that we work over $\CC$, and thus only consider an E-structure. First recall from \cite[Lemma~6.3]{Hug} the decomposition of E-structures of $QH^*(X)[[u]]$ by eigenvalues of $c_1 \star: QH^*(X) \rightarrow QH^*(X)$: \begin{equation}
	QH_*(X)[[u]] = \bigoplus_w QH_*(X)[[u]]_w.
\end{equation}
This decomposition arose as a consequence of general theory of differential equations considered by Hukuhara \cite{Huk}, Levelt \cite{Lev}, Turrittin \cite{Tur} and Wasow \cite{Was}. It was then shown (see \cite[Corollary~6.5]{Hug}) that \begin{equation}
	\OC^-(HC^-(Fuk_w(X))) \subset QH_*(X)[[u]]_w,
\end{equation}
where $Fuk_w(X)$ denotes the monotone Fukaya category with objects Lagrangians, equipped with local systems, with Maslov index $2$ disk count equal to $w$. By construction, for a Lagrangian brane $L \in Fuk_w(X)$, relative cyclic homology takes the form:
 \begin{equation}
	0 \rightarrow HC^-_*(Fuk_w(X)) \rightarrow HC^-_*(Fuk_w(X),L) \rightarrow \E^{-\frac{w}{u}}[-1] \rightarrow 0,
\end{equation}
where $\E^{-\frac{w}{u}}[-1]$ is the E-structure given by $\CC[[u]]$ with connection $\nabla_{\frac{d}{du}} = \frac{d}{du} -\frac{1}{u} + \frac{w}{u^2}$. Now consider relative quantum homology for such $L$. It follows from the construction (see Section \ref{section: the relative quantum TE-structure}) that the eigenvalue decomposition of $QH_*(X,L)[[u]]$ takes the form \begin{equation}
	QH_*(X)[[u]]_{w' \neq w} \cong QH_*(X,L)_{w' \neq w},
\end{equation}
and an exact sequence: \begin{equation}
 0 \rightarrow QH_*(X)[[u]]_w \rightarrow QH_*(X,L)[[u]]_w \rightarrow \E^{-\frac{w}{u}} \rightarrow 0.
 \end{equation}
In particular, if $L$ is a Lagrangian brane with $w$ not an eigenvalue of $c_1 \star$ (which implies $L$ is a zero object in $Fuk(X)$), then $QH_*(X,L)[[u]] \cong QH_*(X)[[u]] \oplus \E^{-\frac{w}{u}}$. Thus in this case the extension is trivial.
Finally, we obtain from \cite[Lemma~2.14]{Hug}: \begin{nlemma}The relative cyclic open-closed map satisfies:
	$$\OC^-_L(HC^-_*(Fuk_w(X),L)) \subset QH_*(X,L)[[u]]_w.$$
\end{nlemma}
\subsection{Outline}
In section \ref{section: relative cyclic homology} we construct relative cyclic homology. In section \ref{section: technical setup} we state the technical assumptions necessary for our constructions and recall the definitions of the $\q$-operations used by Solomon-Tukachinsky. In section \ref{section: the relative quantum TE-structure} we modify the construction of relative quantum cohomology by \cite{ST2} to incorporate the equivariant parameter $u$, and additionally construct a connection in the $u$-direction. In section \ref{section: the relative cyclic open closed map} we construct the relative cyclic open-closed map and prove that it respects the connections. Finally, in section \ref{section: extensios of VHS} we show how to obtain open Gromov-Witten invariants from the Fukaya category of a Calabi-Yau. There, we also prove a classification result for extensions of VHS.
\subsection{Acknowledgements}
	I am greatly indebted to my advisor, Nick Sheridan, who explained to me the initial idea of the relative cyclic open-closed map. I would also like to thank Sara Tukachinsky for explaining a wide variety of ideas from her series of joint papers with Jake Solomon. This work was partly supported by the Royal Society Research Grant for Research Fellows: RGFnR1n181009ERC, the ERC Starting Grant 850713 – HMS and EPSRC Grant EP/W015749/1.
	\section{Relative cyclic homology}
	\label{section: relative cyclic homology}
	
	Let $\MC$ and $\MD$ be strictly unital, uncurved, $R$-linear $A_\infty$-categories. Let $F: \mathcal{C} \rightarrow \mathcal{D}$ be a unital, $R$-linear $A_{\infty}$-functor. This induces a map on Hochschild chains: \[
	F_*: CC_*(\mathcal{C}) \rightarrow CC_*(\mathcal{D}).
	\]
	We then define: \begin{defi}
	The \emph{Hochschild cone-complex} of $F$ is given by the cone of $F_*$: \[
		CC_*(F) := CC_*(\mathcal{C})[1] \oplus CC_*(\mathcal{D}).
		\]
	The differential is thus given by $d(\alpha, \beta) := (b_{\MC}(\alpha), b_{\MD}(\beta) +(-1)^{|\al|}F_*(\alpha))$. Denote the homology of this complex by $HH_*(F)$. Similarly, we can consider the (negative) cyclic cone-complex $CC_*^-(F)$ given by: \[
	CC^-_*(F) := CC^-_*(\mathcal{C})[1] \oplus CC^-_*(\mathcal{D}).
	\]
	Here the differential is given by $d_F^-(\alpha, \beta) = (d^-_\MC(\alpha), d^-_\MD(\beta) + (-1)^{|\al|}F_*(\alpha))$, where $d^-_{\MC} = (b_{\MC} + uB_{\MC})$ and similarly for $d^-_{\MD}$. Let $HC^-_*(F)$ denote the homology of this complex, and call it the \emph{relative negative cyclic homology}.
	\end{defi}
	\begin{remark}
		The signs here are chosen so that $d_F^-$ satisfies $d_F^-(f\al,f\beta) = (-1)^{|f|}fd_F^-(\al,\beta)$ and squares to zero.
	\end{remark}
	\subsection{The relative Getzler-Gauss-Manin connection}
	We will now define a relative notion of the Getzler-Gauss-Manin connection on $HC^-_*(F)$. For the construction of the absolute connection we refer to \cite{Ge}, or \cite[Definition~3.16]{Hug} for a definition using the same setup as used here. 
	
	\begin{notation}
		Let $S_n[k]$ be the set of all partitions of $\lbrace 1, \dots k\rbrace$ into $n$ ordered sets of the form $(1, 2, \dots, k_1)$, $(k_1 + 1, \dots, k_1 + k_2), \dots ,(k_1 + \dots + k_{n-1} + 1, \dots, k_1 + \dots + k_n)$. Let $(i:n)$ denote the $i$th set of the partition. The size of $(i:n)$ is $k_i$. We allow for the case $k_i = 0$. For $\alpha = (\alpha_1,\dots,\alpha_k)$, let $\al^i$ denote the tuple $\alpha^{(i:n)}$, and set $\epsilon(\alpha) := \sum_{j=1}^k |\alpha_j|'$. For a partition $P \in S_n[k]$ let $\epsilon_i = \epsilon(\alpha^{i})$.
	\end{notation}

	Sheridan \cite[Appendix~B]{She} shows that $F_*: HC^-_*(\MC) \rightarrow HC^-_*(\MD)$ respects the Getzler-Gauss-Manin connection by defining a homotopy $H: Der_\KK R \otimes CC^-_*(\MC) \rightarrow u^{-1}CC^-_*(\MD)$. It is given by: $H = \frac{H^2 - H^1}{u} + H^3$, where for $v \in Der_\KK R$ we have: \begin{align}
		H^1_v(\al) &= \sum_{\substack{s \leq k \\ P \in S_s[k]}} (-1)^{\star_1} F^*\left(\al^{s-2},v(\m_\MC)(\al^{s-1}),\al^s, \al_0, \al^1 \right)[F^*(\al^2)|\dots|F^*(\al^{s-3})]\\
		%H^2_v(\al) &= \sum_{\substack{s \leq k \\ s_1 \leq s_2 \leq s_3 \leq s \\ P \in S_s[k]}} (-1)^{\star_2} \m_\MD \left(F^*(\al^{s_2+1}), \dots, v(F^*)(\al^{s_3}), F^*(\al^{s_3+1}), \dots, F^*(\al^s, \al_0, \al^1), \dots, F^*(\al^{s_1})\right) \notag\\
		%&\qquad \qquad[F^*(\al^{s_1 +  1})| \dots |F^*(\al^{s_2})]\\
		H^2_v(\al) &= \sum_{\substack{s \leq k \\ P \in S_s[k]}} (-1)^{\star_2} \m_\MD \left(F^*(\dots), \dots, v(F^*)(\dots), F^*(\dots), \dots, F^*(\al^s, \al_0, \al^1), \dots, F^*(\dots)\right) [F^*(\dots)| \dots ]\\
		H^3_v(\al) &= \sum_{\substack{s \leq k \\ P \in S_s[k]}} (-1)^{\star_3} \mathbf{e} \left[ F^*(\dots)| \dots| v(F^*)(\dots)|\dots|F^*(\al^s, \al_0,\al^1)|\dots \right]
	\end{align}
	The signs $\star_i$ are those arising from permuting elements of $\al$ to obtain the expression for $H_i$, together with signs coming from passing $\m$ or $v$ through $\al$. For example: \begin{equation}
		\star_1 = (\epsilon_{s-2} + \epsilon_{s-1} + \epsilon_{s})(|\al_0|' + \epsilon_1 + \dots + \epsilon_{s-3}) + \epsilon_{s-2}|v|'.
	\end{equation} 
	Sheridan then proves: \begin{nthm}[{\cite[Theorem~B.2]{She}}]
		\label{thm: Sheridan homotopy}
		The homotopy $H$ satisfies: $$F_* \circ \nabla^\MC_v - \nabla^\MD_v \circ F_* = H_v \circ d^-_\MC + (-1)^{|v|}d^-_\MD \circ H_v.$$ Thus, $F_*: HC^-_*(\MC) \rightarrow HC^-_*(\MD)$ respects the Getzler-Gauss-Manin connection.
	\end{nthm}
\begin{remark}
	Sheridan works in a setting where $|v| = 0$ for all $v \in Der_\KK R$. The above result follows trivially by incorporating deformations of odd degree.
\end{remark}
We then define:
	\begin{defi}
		For $v \in Der_\KK R$ define the connection $\nabla^F_v: CC^-_*(F) \rightarrow u^{-1} CC^-_*(F)$ by:
		\[
		\nabla^F_v(\al,\beta) = (\nabla^\MC_v \al, (-1)^{|\al|}H_v(\al) + \nabla^\MD_v \beta ).
		\]
	\end{defi}
	\begin{remark}
		One can check from the definition that the homotopy $H_v$ satisfies $H_v(f\al) = (-1)^{|v|'|f|}f H_v(\al)$. This ensures $\nabla^F_v$ satisfies the Leibniz rule. 
	\end{remark}
\begin{nlemma}
	\label{lem: relative connection descends to cyclic homology}
	The relative connection satisfies $[\nabla_v^F, d^-_F] = 0$ and thus descends to a connection on relative cyclic homology.
\end{nlemma}
\begin{proof}
	We have: \[
	\nabla^F_v(d^-_F(\al,\beta)) = \left(\nabla_v^\MC(d_\MC^-\al), (-1)^{|\al|'}H_v(d_\MC^-(\al)) + (-1)^{|\al|}\nabla_v^\MD(F_*(\al)) + \nabla_v^\MD(d^-_\MD(\beta))\right).
	\]
	Furthermore, \[
	d^-_F(\nabla_v^F(\al,\beta)) = \left( d^-_\MC(\nabla_v^\MC(\al)), (-1)^{|\al|+|v|}F_*(\nabla_v^\MC(\al)) + (-1)^{|\al|}d^-_\MD(H_v(\al)) + d^-_\MD(\nabla_v^\MD(\beta))   \right).
	\]
	The result then follows from Theorem \ref{thm: Sheridan homotopy}, and the definition of the supercommutator: $[\nabla_v^F,d^-_F] = \nabla_v^F \circ d^-_F - (-1)^{|v|} d^-_F \circ \nabla_v^F$.
\end{proof}
We expect:
\begin{ncon}
	\label{con: relative connection is flat}
	The relative connection is flat, thus endowing $HC^-_*(F)$ with a T-structure.
\end{ncon}
Unfortunately, we were not able to verify this. For our main application, Example \ref{eg: relative cyclic chains}, the homotopy $H = 0$, and the conjecture follows readily. For the rest of this section, we will assume flatness has been proven.
\subsection{Euler-gradings}
Recall the definition of an Euler-graded $A_\infty$-category from \cite[Definition~3.18]{Hug}:
\begin{defi}
	\label{Euler-grading A infinity category}
	An Euler-grading on an $A_\infty$-category $\MC$ consists of an Euler vector field $E \in Der_\KK R$ of even degree and, for all objects $C_1, C_2 \in \MC$, even degree maps $Gr: hom_\MC(C_1,C_2) \rightarrow hom_\MC(C_1,C_2)$ such that \begin{equation}
		Gr \circ \m_k = \m_k \circ Gr + (2-k) \m_k.
	\end{equation}
	and \begin{equation}
		Gr(f\al) = 2E(f)\al + fGr(\al) \text{ for } f\in R \text{ and } \al \in hom_\MC(C_1,C_2). 
	\end{equation}
	Furthermore, we require that $Gr(\mathbf{e}) = 0$. An $R$-linear $A_\infty$-functor $F: \MC \rightarrow \MD$ is said to be Euler-graded if $E_\MC = E_\MD$ and $Gr_\MD \circ F^k = F^k \circ Gr_\MC + (1-k)F^k$.
\end{defi} 
The grading operator extends to cyclic homology. Consider $Gr$ as a length-1 Hochschild cochain. Then define the operator $Gr^{-}: CC^-_*(\MC) \rightarrow CC_*^-(\MC)$ by \begin{equation}
	Gr^{-} := \mathcal{L}_{Gr} + \Gamma + 2 u \frac{\partial}{\partial u},
\end{equation}
where $\Gamma(\alpha_0[\alpha_1|\dots|\alpha_k]) = -k \alpha_0[\alpha_1|\dots|\alpha_k]$ is the length operator on cyclic chains, and $\mathcal{L}$ is the curvature operator, introduced in \cite{Ge}, see  \cite[Equation~77]{Hug} for a definition in our setup.
 
Now let $\MC$, $\MD$ be Euler-graded, $R$-linear $A_\infty$-categories, and $F: \MC \rightarrow \MD$ an Euler-graded functor.
\begin{defi}
	\label{defi: grading shift}
Define the grading on the negative cyclic cone-complex by: \begin{align}
		Gr_F^-: CC^-_*(F) &\rightarrow CC_*^-(F)\\
		(\al,\beta) &\mapsto (Gr^-_\MC(\al) - \al, Gr_\MD^-(\beta)).
	\end{align}
\end{defi}
The extra term $-\al$ comes from the shift in grading in the cone complex. It also ensures that $Gr^-_F$ satisfies the correct relations:
\begin{nlemma}
	The grading satisfies $[Gr^-_F, d^-_F] = d^-_F$ and thus descends to a grading operator on $HC^-_*(F)$. This endows $HC^-_*(F)$ with an Euler-graded T-structure.
\end{nlemma}
\begin{proof}
	The first statement follows as the functor $F$ is Euler-graded. To show that the resulting T-structure is Euler-graded, we need to show: \begin{equation}
		[Gr^-_F,\nabla^F_v] = \nabla^F_{[2E,v]}.
	\end{equation}
	We compute: \begin{equation}
		[Gr^-_F,\nabla^F_v](\al,\beta) = \left(\nabla^\MC_{[2E,v]}(\al),(-1)^{|\al|}\left([Gr^-,H_v](\al) + H_v(\al)\right) +\nabla^\MD_{[2E,v]}(\beta) \right).
	\end{equation}
	The result then follows from the following lemma.
\end{proof}
\begin{nlemma}
	The homotopy $H$ satisfies: $[Gr^-, H_v] = H_{[2E,v]} - H_v$.
\end{nlemma}
\begin{proof}
	Comparing powers of $u$, it is sufficient to show: \begin{align}
		[Gr^-, H_v^i] &= H^i_{[2E,v]} + H^i_v \text{ for } i = 1,2\\
		[Gr^-, H_v^3] &= H^3_{[2E,v]} - H^3_v
	\end{align}
We will show this holds for $H^1$, the others follow from a similar computation. For $\al = \al_0[\al_1, \dots, \al_k]$, we have: \begin{multline}
	\label{eq: H respects grading}
	[Gr^-,H^1_v](\al) = [\Gamma, H^1_v](\al) + \sum_{\substack{s \leq k \\ P \in S_s[k]}} (-1)^{\star_1} [Gr,F^*]\left(\al^{s-2},v(\m_\MC)(\al^{s-1}),\al^s, \al_0, \al^1 \right)[F^*(\al^2)|\dots|F^*(\al^{s-3})]\\
	+\sum_{\substack{s \leq k \\ P \in S_s[k]}} (-1)^{\star_1} F^*\left(\al^{s-2},v(\m_\MC)(\al^{s-1}),\al^s, \al_0, \al^1 \right)[F^*(\al^2)|\dots|[Gr,F^*](\dots)|\dots|F^*(\al^{s-3})]
\end{multline}
First note that: \begin{equation}
	[\Gamma, H^1_v](\al) = \sum_{\substack{s \leq k \\ P \in S_s[k]}} (-1)^{\star_1} (k-s+4) F^*\left(\al^{s-2},v(\m_\MC)(\al^{s-1}),\al^s, \al_0, \al^1 \right)[F^*(\al^2)|\dots|F^*(\al^{s-3})].
\end{equation}
The second term satisfies: \begin{multline}
	\sum_{\substack{s \leq k \\ P \in S_s[k]}} (-1)^{\star_1} [Gr,F^*]\left(\al^{s-2},v(\m_\MC)(\al^{s-1}),\al^s, \al_0, \al^1 \right)[F^*(\al^2)|\dots|F^*(\al^{s-3})]\\
	= \sum_{\substack{s \leq k \\ P \in S_s[k]}} (-1)^{\star_1} (1- (k_{s-2} + 1 + k_s+ 1 + k_1)) F^*\left(\al^{s-2},v(\m_\MC)(\al^{s-1}),\al^s, \al_0, \al^1 \right)[F^*(\al^2)|\dots|F^*(\al^{s-3})]\\
	+ \sum_{\substack{s \leq k \\ P \in S_s[k]}} (-1)^{\star_1} F^*\left(\al^{s-2},[Gr,v(\m_\MC)](\al^{s-1}),\al^s, \al_0, \al^1 \right)[F^*(\al^2)|\dots|F^*(\al^{s-3})].
\end{multline}
Now note that:\begin{align}
	[Gr,v(\m_\MC)](\al^{s-1}) &= [Gr,v](\m_\MC)(\al^{s-1}) +v([Gr,\m_\MC])(\al^{s-1})\notag \\
	&= [2E,v](\m_\MC)(\al^{s-1}) + (2-k_{s-1})v(\m_{\MC})(\al^{s-1}).
\end{align}
Finally, the third term of \eqref{eq: H respects grading} satisfies: \begin{multline}
	\sum_{\substack{s \leq k \\ P \in S_s[k]}} (-1)^{\star_1} F^*\left(\al^{s-2},v(\m_\MC)(\al^{s-1}),\al^s, \al_0, \al^1 \right)[F^*(\al^2)|\dots|[Gr,F^*](\dots)|\dots|F^*(\al^{s-3})]\\
	= \sum_{\substack{s \leq k \\ P \in S_s[k]}} (-1)^{\star_1} \left(\sum_{i=2}^{s-3} (1-k_i)\right) F^*\left(\al^{s-2},v(\m_\MC)(\al^{s-1}),\al^s, \al_0, \al^1 \right)[F^*(\al^2)|\dots|F^*(\al^{s-3})].
\end{multline}
Combining all the above equations shows the required result.
	\end{proof}
\subsection{The canonical relative connection}
As shown in \cite[Lemma~3.21]{Hug}, associated to any $R$-linear $A_\infty$-category $\MC$ is an $R[s,s^{-1}]$-linear Euler-graded category $\MC^s$, with $\MC^s \otimes_{R[s,s^{-1}]} R = \MC$. Here $R$ becomes an $R[s,s^{-1}]$-module under the map $s \mapsto 1$. Any functor $F: \MC \rightarrow \MD$ canonically determines an Euler-graded functor $F^s: \MC^s \rightarrow \MD^s$. The construction from the previous section then equips $HC^-_*(F^s)$ with an Euler-graded T-structure. As noted in \cite[Definition~2.6]{Hug} this gives rise to a natural connection $\nabla^{F^s}_{\partial_u} = \frac{Gr^-_{F^s}}{2u} - u^{-1}\nabla^{F^s}_E$. By restricting to $s=1$, we obtain a natural connection $\nabla^{F}_{\partial_u}$ on $HC^-_*(F)$. Call this the canonical relative connection. 

\begin{eg}
	\label{eg: relative cyclic chains}
	Let $\mathcal{C} = \mathcal{K}$, the $A_{\infty}$-category with one object, $\star$, and $Hom_{\MK}(\star,\star) = \mathbb{K}$. There is an isomorphism of E-structures $CC^-_*(\MK) \cong HC_*^-(\mathcal{K}) \cong (\mathbb{K}[[u]], \frac{d}{du})$.
	Let $\MD$ be any uncurved $A_\infty$-category. A functor $F: \MK \rightarrow \MD$ is then given by $F(\star) = L$ for some object $L \in ob\MD$ and on morphisms: $F(e) = e_L \in hom(L,L)$. In this case, we denote the relative cyclic chain complex by $CC^-_*(\MD,L)$, with associated relative cyclic homology $HC^-_*(\MD,L)$. Now, the functor $F$ respects the connection on the chain level, and the homotopy $H$ vanishes, which shows that Conjecture \ref{con: relative connection is flat} holds. On the chain level, the canonical relative connection thus takes the particularly simple form: \begin{equation}
		\nabla^{L}_{\frac{d}{du}}(a,\alpha) = ( \frac{da}{du} -\frac{a}{2u}, \nabla^{\MD}_{\frac{d}{du}} \alpha) \text{ for } (a,\al) \in \mathbb{K}[[u]] \oplus CC^-_*(\MD). 
	\end{equation}
	The `extra' term $\frac{a}{2u}$ comes from the shift in grading on the cone complex, see Definition \ref{defi: grading shift}. Note also that even though the connection splits at the chain level, this does not necessarily mean that the relative cyclic homology splits as an E-structure. Instead, we get the following long exact sequence of E-structures:
	\begin{equation}
	\cdots \rightarrow \mathbb{K}[[u]] \xrightarrow{\mathbf{e}_L} HC^-_{*}(\MD) \rightarrow HC_*^-(\MD,L) \rightarrow \cdots
	\end{equation}
	Now assume additionally that $\lbrack \mathbf{e}_L \rbrack = 0 \in HC^-_*(\MD)$. Then, the long exact sequence breaks up into the short exact sequence: \begin{equation}
		\label{SES}
	0 \rightarrow HC_*^-(\MD) \rightarrow HC_*^-(\MD,L) \rightarrow \mathbb{K}[[u]]_{*+1} \rightarrow 0.
\end{equation}
	In particular there is a short exact sequence of E-structures:
	\begin{equation}
	0 \rightarrow HC_{odd}^-(\MD) \rightarrow HC_{odd}^-(\MD,L) \rightarrow \mathbb{K}[[u]] \rightarrow 0.
\end{equation}
	And an isomorphism: \begin{equation}
	HC_{even}^{-}(\MD) \cong HC_{even}^{-}(\MD,L).
\end{equation}
\end{eg}
\begin{eg}
	We can still make the above construction work if the object $L$ in $\MD$ is weakly curved. Suppose that $L$ has curvature $\m_0^L = c \cdot \mathbf{e}_L$, for some $c \in \KK$. Then, consider the category $\mathcal{K}$ to be equipped with the curvature $\m_0^\MK = c \cdot 1$. The connection on $HC^-_*(\MK) = \mathbb{K}[[u]]$ is then given by $\nabla^{\MK}_{\frac{d}{du}} = \frac{d}{du} + \frac{c}{2u^2}$. As \begin{equation}
		\nabla_{\frac{d}{du}} \bfe_L = \frac{c}{u^2}\bfe_L
	\end{equation}
	the map $F_*: CC^-_*(\MK) \rightarrow HC^-_*(\MC)$ respects connections at the chain level. This is necessary, as we find that the homotopy $H = 0$ in this case. The relative $u$-connection can thus be expressed on the chain level as:
	\begin{equation}
	\nabla^{L}_{\frac{d}{du}}(a,\alpha) = ( \frac{da}{du} -\frac{a}{2u} + \frac{c a}{u^2} , \nabla^{\MD}_{\frac{d}{du}} \alpha).
	\end{equation}	
	The exact sequences from before then still hold.
\end{eg}
\begin{eg}
	Now consider the general case of an $R$-linear $A_\infty$-category $\MD$, and an object $L$ with curvature $\m_0^L = c \cdot \mathbf{e}_L$, for some $c \in R$. As the curvature on $\MK$ is $c \cdot 1$, the connection on $HC^-_*(\MK) = \mathbb{K}[[u]]$ is given by $\nabla^{\MK}_{v}(a) = v(a) - \frac{v(c)}{u}a$. The relative Getzler-Gauss-Manin connection can be expressed on the chain level as:
	\begin{equation}
		\nabla^L_v(a,\alpha) = (v(a) - \frac{v(c)}{u}a, \nabla^{\MD}_v\alpha).
	\end{equation}
\end{eg}
	\section{Technical setup}
	\label{section: technical setup}
	\subsection{Regularity assumptions}
	Let $X$ be a $2n$-dimensional symplectic manifold and $J$ be an $\omega$-tame almost complex structure on $X$. Let $L \subset X$ be an oriented Lagrangian equipped with a relative spin structure $\mathfrak{s}$. For us a relative spin structure comes with a choice of element $w_{\mathfrak{s}} \in H^2(X;\ZZ/2)$ such that $w_{\mathfrak{s}}|_L = w_2(TL) \in H^2(L;\ZZ/2)$.
	
	For $l \geq 0$, let $\mathcal{M}_{l+1}(\beta)$ be the moduli space of stable $J$-holomorphic spheres with $l+1$ marked points in homology class $\beta \in H_2(X,\mathbb{Z})$. Let \begin{equation}
		ev_j^\beta: \mathcal{M}_{l+1}(\beta) \rightarrow X
	\end{equation}
	be the evaluation map at the $j$'th marked point.
	For $k\geq -1$, $l\geq 0$, let $\mathcal{M}_{k+1,l}(\beta)$ be the moduli space of $J$-holomorphic stable maps $(\mathbb{D},S^1) \rightarrow (X,L)$ in homology class $\beta \in H_2(X,L)$ with one boundary component, $k+1$ anti-clockwise ordered boundary marked points, and $l$ interior marked points. Let \begin{equation}
		evb_i^\beta: \mathcal{M}_{k+1,l}(\beta) \rightarrow L \text{ and } evi_j^\beta: \mathcal{M}_{k+1,l}(\beta) \rightarrow X
	\end{equation}
	be the evaluation maps at the $i$'th boundary and $j$'th interior marked points respectively. The relative spin structure determines an orientation on the moduli spaces $\mathcal{M}_{k+1,l}(\beta)$, see \cite[Chapter~8]{FOOO}.
	
	We will also need a moduli space of disks with a horocyclic constraint. Recall that a horocycle in a disk is given by a circle tangent to the boundary. These moduli spaces are similar to the ones used in \cite[Chapter~3]{ST2}, where some of the marked points are constrained to lie on a geodesic in $\mathbb{D}$. Our definition is entirely analogous, except that we replace `geodesic' with `horocycle'. Let the smooth locus of $\mathcal{M}_{k+1,l; \perp_i}(\beta) \subset \mathcal{M}_{k+1,l}(\beta)$ be the subset defined by requiring the first and second interior marked points $w_1$ and $w_2$ to lie at $-t$ and $t$ respectively for $t \in (-1,1)$ and fixing the i'th boundary point $z_i$ at $-i$. Equivalently, we require that $z_i$, $w_1$, $w_2$ lie on a horocycle in anti-clockwise ordering. This moduli space also appeared in \cite{Ga12}, where it was used to show that the closed-open map is an algebra homomorphism.
	
	We now give a more formal definition of the moduli space $\mathcal{M}_{k+1,l; \perp_i}(\beta)$ as a fibre product of known spaces. Consider the forgetful map $\mathcal{M}_{k+1,l}(\beta) \rightarrow \MM_{1,2} = D^2$, only remembering the i'th boundary marked point, and the first two interior marked points. Here the identification $\MM_{1,2} \cong D^2$ is achieved by using an automorphism of the disk to map the boundary marked point to $-i$, and the first interior marked point to $0$. Consider the inclusion $I \hookrightarrow D^2$ given by the arc of the horocycle through $-i$ and $0$ with negative real part. This is a circle of radius $\frac{1}{2}$ centred at $-\frac{i}{2}$. The condition on the order of the marked points means that second interior point lies on the semi-circle with negative real part. We then define: \begin{equation}
		\mathcal{M}_{k+1,l; \perp_0}(\beta) = I \times_{D^2} \mathcal{M}_{k+1,l}(\beta).
	\end{equation}
	
	Take the orientation on $I$ to be the positive orientation, so that $\partial I = \{1\} -\{ 0\}$. The orientation on $\mathcal{M}_{k+1,l; \perp_0}(\beta)$ is then defined by the fibre-product orientation, as in \cite[Section 2.2]{ST4}.

	We assume the following:
	\begin{assumptions}
		$\text{ }$
		\label{assumptions}
		\begin{enumerate}
			\item $\mathcal{M}_{l+1}(\beta)$ is a smooth orbifold with corners.
			\label{assumptions 1}
			\item $ev_0$ is a proper submersion.
			\label{assumptions 2}
			\item $\mathcal{M}_{k+1,l}(\beta)$ is a smooth orbifold with corners.
			\label{assumptions 3}
			\item $evb_0^\beta$ is a proper submersion.
			\label{assumptions 4}
			\item $\mathcal{M}_{k+1,l; \perp_i}(X,\beta)$ is a smooth orbifold with corners.
			\label{assumptions 5}
			\item $evb_0^\beta|_{\mathcal{M}_{k+1,l; \perp_i}(X,\beta)}$ is a proper submersion.
			\label{assumptions 6}
		\end{enumerate}
	\end{assumptions}
	As was noted in \cite[Section~1.3.12]{ST2} for assumptions 1-4, and in \cite[Lemma~4.19]{Hug} for assumptions 5-6, we have:
	\begin{nlemma}
		\label{G action assumptions hold}
		The above assumptions hold for $L \subset X$ a Lagrangian and a complex structure $J$ with the following properties:
		\begin{itemize}
			\item $J$ is integrable.
			\item There exists a Lie group $G_X$ acting J-holomorphically and transitively on $X$.
			\item There exist a Lie subgroup $G_L \subset G_X$ whose action restricts to a transitive action on $L$.
		\end{itemize}
	\end{nlemma}
\begin{eg}
	The simplest example is $(\mathbb{CP}^n, T_{cl})$, where $T_{cl}$ denotes the Clifford torus. Here the groups acting are $G_X = SU(n+1)$ and $G_L = T^{n}$. Other examples (see {\cite[Example~1.5]{ST3}}) are $(\mathbb{CP}^n, \mathbb{RP}^n)$, or more generally flag varieties and Grassmannians with $L$ being the real locus. In these cases $G_X$ is a complex matrix group and $G_L$ the real subgroup. Another class of examples are the quadric hypersurfaces with real locus $S^n$: \begin{equation}
		X_{2,n} = \big\{ \sum_{i=0}^n z_i^2 = z_{n+1}^2 \big\} \subset \mathbb{CP}^{n+1}.
	\end{equation}
	In this case, $G_X = SO(n+1,1, \CC)$ and $G_L = SO(n+1,1, \RR)$, with the groups acting in the obvious manner. 
\end{eg}
	\subsection{The $\q$-operations}
	First we introduce some coefficient rings. Consider the Novikov ring: \begin{align}
		\Lambda &= \left\{  \sum_{i = 0}^{\infty} a_iQ^{\lambda_i} | a_i \in \mathbb{C}, \; \lambda_i \in \mathbb{R}_{\geq 0},\; \lim_{i\to \infty } \lambda_i = \infty \right\}.
	\end{align}
	To account for gradings, we will also make use of the `universal Novikov ring': $\Lambda^{e} := \Lambda[e,e^{-1}]$, where $e$ has degree $2$.
	\begin{remark}
		A lot of our work is based on \cite{ST3}. They use a different Novikov ring, more commonly used in Gromov-Witten theory. Instead of taking series in $Q^{\mathbb{R}}$ they take series with terms $T^{\beta}$ for $\beta \in H_2(X,L)$. For them the monomial $T^\beta$ has degree $\mu(\beta)$, where $\mu: H_2(X,L) \rightarrow \mathbb{Z}$ is the Maslov index. The graded map $T^{\beta} \mapsto Q^{\omega(\beta)}e^{\mu(\beta)/2}$, allows us to compare their Novikov ring with the universal Novikov ring $\Lambda^e$. Note that $\mu(\beta) \in 2\mathbb{Z}$ as we assume our Lagrangian is orientable.
	\end{remark}

	Let $L \subset X$ be a Lagrangian satisfying assumptions \ref{assumptions}. Suppose $L$ is equipped with a $\Lambda^*$-local system, with monodromy representation: \begin{equation} hol: H_1(L,\mathbb{Z}) \rightarrow \Lambda^*. \end{equation}
	
	Let $A^*(L)$ denote differential forms on $L$ with coefficients in $\mathbb{C}$, and similarly for $X$.
	For $\al \in A^*(L)$, let $|\al|$ denote its degree as a differential form, and similarly for differential forms on $X$. Also, let $|\al|' := |\al| - 1$, and for an ordered set $\al = (\al_1, \dots, \al_k)$, write $\epsilon(\al) := \sum_i |\al_i|' \in \mathbb{Z}$.
	
	For $k,l \geq 0$ and $\beta \in H_2(X,L)$ with $(k,l,\beta) \notin \lbrace (1,0, \beta_0), (0,0,\beta_0) \rbrace$,\cite{ST3} define operations: \begin{equation}
		\mathfrak{q}_{k,l}^{ST,\beta}:A^*(L)^{\otimes k} \otimes A^*(X)^{\otimes l} \rightarrow A^*(L).
	\end{equation}
	We extend their definition to take into account the local system and set: \begin{equation}
		\mathfrak{q}_{k,l}^{ST,\beta}(\alpha_1 \otimes \dots \otimes \alpha_k; \gamma_1 \otimes \dots \otimes \gamma_l) := (-1)^{\zeta(\alpha)} hol(\partial \beta) (evb_0^\beta)_* \bigg{(} \bigwedge_{j=1}^l (evi^\beta_j)^* \gamma_j \wedge \bigwedge_{i=1}^k (evb_i^\beta)^* \alpha_j \bigg{)}. 
	\end{equation}
	Here $\zeta(\alpha) = 1 + \sum_{j=1}^{k} j|\alpha_j|'$. The special cases are as follows: 
	\begin{align}
		&\mathfrak{q}_{0,0}^{ST,\beta} := - (evb_0^\beta)_* 1 \in A^*(L),\\
		&\mathfrak{q}_{1,0}^{ST,\beta_0} := d\alpha,\\
		&\mathfrak{q}_{0,0}^{ST,\beta_0} := 0.
	\end{align}
	
	For the cleanest staments, we will use a sign convention differing from \cite{ST3}.
	\begin{defi}
		Let the operations $\q_{k,l}^\beta: A^*(L)^{\otimes k} \otimes A^*(X)^{\otimes l} \rightarrow A^*(L;\Lambda)$ be defined by: \begin{equation}
			\q_{k,l}^\beta(\al_1, \dots, \al_k;\ga_1, \dots, \ga_l) = (-1)^{\dagger(\al) + k-1} \q_{k,l}^{ST,\beta}(\al_k, \dots, \al_1;\ga_1, \dots, \ga_l).
		\end{equation}
		Here, for $\al = (\al_1, \dots, \al_k)$, we set $\dagger(\al) = \sum_{1 \leq i <j \leq k} |\al_i|'|\al_j|'$. This is the sign coming from reversing the order of $\al$. 
	\end{defi}
	
	\cite{ST2} also define closed operations:
	\begin{equation}
		\mathfrak{q}_{\emptyset, l}^{ST,\beta}: A^*(X)^{\otimes l} \rightarrow A^*(X),
	\end{equation}
	by
	\begin{equation}
		\mathfrak{q}_{\emptyset, l}^{ST,\beta}(\gamma_1, \dots, \gamma_l) := (-1)^{w_{\mathfrak{s}}(\beta)} (ev_0^\beta)_*(\bigwedge_{j=1}^l (ev_j^\beta)^* \gamma_j),
	\end{equation}
	with special cases: \begin{equation}
		\mathfrak{q}_{\emptyset, 1}^{ST,\beta_0} := 0, \; \mathfrak{q}_{\emptyset, 0}^{ST, \beta_0} := 0.
	\end{equation}
	We use these operations, without any sign change, so that $\q_{\emptyset,l} = \q^{ST}_{\emptyset,l}$. 
	
	We also recall from \cite{Hug} the definition of operations coming from the moduli spaces with horocyclic constraints $\MM_{k+1,l,\perp_i}(\beta)$. We first define these using sign conventions similar to \cite{ST3}.
	\begin{defi}
		Let $\mathfrak{q}^{ST,\beta}_{k,l;\perp_i}: A^*(L)^{\otimes k} \otimes A^*(X)^{\otimes l} \rightarrow A^*(L;\Lambda),$
		be defined by 
		\begin{equation}
			\mathfrak{q}^{ST,\beta}_{k,l;\perp_i}(\alpha_1, \dots, \alpha_k;\gamma_1, \dots, \gamma_l) = (-1)^{\zeta(\alpha) + {\zeta_{\perp}(\alpha;\gamma)}}hol(\partial \beta) (evb_0^\beta)_*\big{(}\bigwedge_{j=1}^{l} (evi_j^{\beta})^* \gamma_j \wedge \bigwedge_{j=1}^{k} (evb_j^{\beta})^*\alpha_j\big{)},
		\end{equation}
		where $\zeta_{\perp}(\alpha;\gamma) = |\alpha|' + |\gamma| + n$. Then, we modify the sign convention as before. We set
		\begin{equation}
			\q_{k,l,\perp_i}^\beta(\al_1, \dots, \al_k;\ga_1, \dots, \ga_l) = (-1)^{\dagger(\al) + k-1} \q_{k,l,\perp_{k+1-i}}^{ST,\beta}(\al_k, \dots, \al_1;\ga_1, \dots, \ga_l).
		\end{equation}
		The sign $\dagger(\al)$ is as before. When $i = 0$, $\perp_{k+1}$ should be interpreted as $\perp_0$.
	\end{defi}
	
	For all of the above $\mathfrak{q}_{*}^\beta$ operations, set \begin{equation}
		\mathfrak{q}_{*} = \sum_\beta Q^{\omega(\beta)}e^{\frac{\mu(\beta)}{2}} \mathfrak{q}^\beta_{*}.
	\end{equation}
	Here $\mu: H_2(X,L) \rightarrow \mathbb{Z}$ is the Maslov-class, and $e$ is of degree $2$. We thus have operations \begin{equation}
		\mathfrak{q}_{k,l}:A^*(L;\Lambda^e)^{\otimes k} \otimes A^*(X;\Lambda^e)^{\otimes l} \rightarrow A^*(L;\Lambda^e).
	\end{equation}
	
	Let $\langle \alpha_1, \alpha_2 \rangle_L = (-1)^{|\alpha_2|}\int_L \alpha_1 \wedge \alpha_2$ be the Poincar\'e pairing on $L$. \cite{ST3} prove results about the operations $\q^{ST}$, we state the analogous results for our operations $\q$. These follows by a direct verification of signs from the results in \cite{ST3}.
	\begin{nprop}[cyclic symmetry, see {\cite[Proposition~3.3]{ST3}}]
		\label{cyclic symmetry}
		For any $\alpha = (\alpha_1, \dots, \alpha_{k+1})$ and $\gamma = (\gamma_1, \dots, \gamma_l)$:
		\begin{equation}
			\langle \mathfrak{q}_{k,l}(\alpha_1, \dots, \alpha_k; \gamma_1, \dots, \gamma_l),\alpha_{k+1} \rangle_L = (-1)^{|\alpha_{k+1}|'\epsilon_k(\alpha)} \langle  \mathfrak{q}_{k,l}(\alpha_{k+1}, \alpha_1, \dots, \alpha_{k-1}; \gamma_1, \dots, \gamma_l),\alpha_{k} \rangle_L
		\end{equation}
	\end{nprop}
	\begin{nprop}[degree property, see {\cite[Proposition~3.5]{ST3}}]
		\label{degree property}
		For any $\alpha = (\alpha_1, \dots, \alpha_{k})$ and $\gamma = (\gamma_1, \dots, \gamma_l)$:
		\begin{align}
			|\mathfrak{q}^\beta_{k,l}(\alpha_1, \dots, \alpha_k; \gamma_1, \dots, \gamma_l)| &= 2 + \epsilon(\alpha) -\mu(\beta) + \sum_{j=1}^l (|\gamma_j| - 2)  \\
			&\equiv \epsilon(\alpha) + \sum_{j=1}^l |\gamma_j| \; (\text{mod } 2)
		\end{align}
		The last equality holds as $L$ is orientable so the Maslov-index of any disk is even. By construction, we then have: \begin{equation}
			|\q_{k,l}(\al,\ga)| = 2 + \epsilon(\al) + |\ga| -2l. 
		\end{equation}
	\end{nprop}
	\begin{nprop}[unit property, see {\cite[Proposition~3.2]{ST3}}]
		\label{unit property}
		For $f \in A^0(L)$, $\alpha_1, \dots, \alpha_k \in A^*(L;R)$ and $\gamma \in A^*(X;Q)^{\otimes l}$, we have:
		\begin{equation}
			\mathfrak{q}^\beta_{k,l}(\alpha_1, \dots, \alpha_{i-1},f,\alpha_i, \dots, \alpha_k; \gamma) =
			\begin{cases}
				df  &(k,l,\beta) = (1,0,\beta_0)\\
				(-1)^{|f|}f\alpha_1 &(k,l,\beta) = (2,0,\beta_0), \; i=1\\
				(-1)^{|\alpha_1||f|'}f\alpha_1 & (k,l,\beta) = (2,0,\beta_0), \; i=2\\
				0 & \text{otherwise}
		\end{cases}\end{equation}
	\end{nprop}
\begin{nprop}[top degree property, see {\cite[Proposition~3.12]{ST3}}]
	\label{top degree property}
	Let $\delta_n$ denotes the degree $n$ part of a differential form $\delta \in A^*(L;R)$. Then $(\mathfrak{q}^\beta_{k,l}(\alpha;\gamma))_n = 0$ for $(k,l,\beta) \notin \{ (1,0,\beta_0), (0,1,\beta_0), (2,0,\beta_0)  \}$.  
\end{nprop}
	\begin{nprop}[divisor property, see {\cite[Proposition~3.9]{ST3}}]
		\label{divisor property}
		Assume $\gamma_1 \in A^2(X,L)$, $d\gamma_1 = 0$, then \begin{equation}
			\mathfrak{q}^\beta_{k,l}(\alpha, \gamma) = (\int_\beta \gamma_1)\cdot \mathfrak{q}^\beta_{k,l-1}(\alpha; \bigotimes_{j \geq 2} \gamma_j)
		\end{equation}
	\end{nprop}
	\begin{nprop}[energy-zero property, see {\cite[Proposition~3.8]{ST3}}]
		\label{energy zero property}
		For $k \geq 0$, 
		\begin{equation}
			\mathfrak{q}_{k,l}^{\beta_0}(\alpha_1, \dots, \alpha_k; \gamma_1, \dots, \gamma_l) = \begin{cases*}
				d\alpha_1, &$(k,l) = (1,0)$,\\
				(-1)^{|\alpha_1|}\alpha_1 \wedge \alpha_2, &$(k,l) = (2,0)$,\\
				\gamma_1|_L, &$(k,l) = (0,1)$,\\
				0, &\text{otherwise}.
			\end{cases*}
		\end{equation}
	\end{nprop}
	\begin{nprop}[fundamental class property, see {\cite[Proposition~3.7]{ST3}}]
		\label{fundamental class property}
		For $k \geq 0$, \begin{equation}
			\mathfrak{q}_{k,l}^{\beta}(\alpha_1, \dots, \alpha_k; 1, \gamma_1, \dots, \gamma_{l-1}) = 
			\begin{cases}
				1, &(k,l,\beta) = (0,1,\beta_0),\\
				0, & \;\text{otherwise}
		\end{cases}\end{equation}
	\end{nprop}
	Let  $\gamma = (\gamma_1, \dots,\gamma_l)$ be a list of differential forms on $X$. For subsets $I \sqcup J = \{ 1, \dots, l\}$, define $sign^\gamma(I,J)$ by the equation \begin{equation}
		\bigwedge_{i\in I} \gamma_i \wedge \bigwedge_{j\in J} \gamma_j = (-1)^{sign^\gamma(I,J)}\bigwedge_{s\in [l]} \gamma_s,
	\end{equation}
	or explicitly \begin{equation}
		sign^\gamma(I,J) = \sum_{i \in I, j \in J, j<i} |\gamma_i||\gamma_j|.
	\end{equation}
	\subsection{The $\q_{-1}$-operations}
	\label{section: q -1 operations}
	Solomon-Tukachinsky \cite{ST3} define counts of disks with no boundary marked points \begin{equation}
		\mathfrak{q}_{-1,l}^{\beta}: A^*(X)^{\otimes l } \rightarrow \mathbb{C},
	\end{equation}
	by setting \begin{equation}
		\mathfrak{q}_{-1,l}^{\beta}(\gamma_1 \otimes \dots \otimes \gamma_l) := \int_{\mathcal{M}_{0,l}(\beta)} \bigwedge_{j=1}^l (evi^\beta_j)^* \gamma_j,
	\end{equation}
	with special cases \begin{equation}
	 \mathfrak{q}_{-1,1}^{\beta_0} := 0, \; \mathfrak{q}_{-1,0}^{\beta_0} := 0.
	\end{equation}
	In contrast to op.\ cit.\ we additionally incorporate the parameter $e$ to keep track of grading, and set: \begin{equation}
	\q_{-1,l} := \sum_{\beta} \q_{-1,l}^\beta Q^{\omega(\beta)}e^{\mu(\beta)}.
\end{equation}
	We will use these operations without any sign changes. The following follows directly from op.\ cit.\ 
	\begin{nprop}[{\cite[Proposition~2.5]{ST3}}]
		\label{boundary of q -1 operation}
		For any $\ga = (\ga_1, \dots, \ga_l)$:
		\begin{align}
			0 &= \sum_{\substack{ P \in S_3[k] \\ (2:3) = \{ j \}}}  (-1)^{|\gamma^{(1:3)}| + 1}\mathfrak{q}_{-1,l}(\gamma^{(1:3)} \otimes d\gamma_j \otimes \gamma^{(3:3)})\nonumber\\
			&+ \frac{1}{2}\sum_{\substack{I \sqcup J = [l]}} (-1)^{i(\gamma, I)} \langle \mathfrak{q}_{0,|I|}( \gamma^I) , \mathfrak{q}_{0,|J|}(\gamma^{J}) \rangle_L + (-1)^{|\gamma|+1} \int_L i^*\mathfrak{q}_{\emptyset, l}(\gamma),
		\end{align}
		where $i(\gamma,I) := \sum_{j \in I} |\ga_j| + sgn^\ga(I,[l]\setminus I)$.
	\end{nprop}
	\begin{nprop}[compare with {\cite[Lemma~3.9]{ST2}}]
		\label{unit on the horocycle 2}
		$\mathfrak{q}_{-1,l}(\gamma) = (-1)^{|\gamma|}\langle \mathfrak{q}_{0,l;\perp_0}(\gamma),1\rangle_L$.
	\end{nprop}
	\begin{proof}
		The proof is very similar to that of \cite[Lemma~3.9]{ST2} and \cite[Proposition 4.32]{Hug}. We use a different sign convention to \cite{ST2} in the $\q$ operations, and so need to modify the sign accordingly. Following through their proof gives the sign: \begin{equation}
			\zeta(\emptyset) + \zeta_{\perp}(\emptyset,\gamma) + n + 1 \equiv |\gamma|\;(mod \; 2).
		\end{equation}
	\end{proof}
	
	\subsection{Bulk-deformations and bounding cochains}
	\label{section: bulk-deformations and bounding cochains}
	For a $\ZZ$-graded $\CC$-vector space $U$, let $\CC[[U]]$ be the ring of formal functions on the completion of $U$ at the origin. Explicitly, let $\{v_i\}_{i \in I}$ be a homogeneous basis for $U$, and $\{v_i^*\}_{i \in I}$ the dual basis for $U^*$. Let $\{t_i\}_{i \in I}$ be formal variables of degree $-|v_i|$, then we have an isomorphism: \begin{align}
		\label{isomorphism coefficient ring}
		\CC[[t_i]]_{i \in I} &\cong \CC[[U]],\notag \\
		t_i &\mapsto v_i^*.
	\end{align}
	Each formal vector field $v \in \CC[[U]] \otimes U$ on $U$ can be viewed as a derivation $\partial_v: \CC[[U]] \rightarrow \CC[[U]]$. In coordinates, if $v = \sum_i f_i v_i$, for some $f_i \in \CC[[U]]$, then $\partial_v = \sum_i f_i \partial_{t_i}$. Define the vector fields \begin{equation}
		\Gamma_U = \sum_i t_i \partial_{t_i} \text{ and } E_U = \sum_i \frac{deg(t_i)}{2}t_i\partial_{t_i}.
	\end{equation}
	These are independent of the chosen basis. For $l \in \mathbb{Z}$, let $U[l]$ denote the graded vector space with $U[l]^i = U^{i+l}$. Then set:
	\begin{equation}
		Q_U :=  \Lambda[[U[2]]] \text{ and } Q_U^e := \Lambda^e[[U[2]]].
	\end{equation}
	
	Following \cite{ST3}, define the valuation $\zeta_Q: Q_U \rightarrow \RR_{\geq 0}$ by: \begin{equation}
		\label{valuation on coefficient ring}
		\zeta_Q\left(\sum_{j = 0}^{\infty} a_jQ^{\lambda_j} \prod_{i \in I} t_i^{l_{ij}}\right) = \inf_{\substack{j \\ a_j\neq 0}} (\lambda_j + \sum_{i \in I} l_{ij}).
	\end{equation} 
	Let $\mathcal{I}_U = \{ f \in Q_U | \zeta_Q(f) > 0 \} \subset Q_U$. We extend the valuation $\zeta_Q$ to $Q_U^e$ by setting $\zeta_{Q^e}(e) = 0$.
	\begin{defi}
		A \emph{bulk-deformation parameter} over $U$ is an element $\gamma \in \mathcal{I}_UA^*(X;Q_U)$ with $d\gamma = 0$, $|\gamma| = 2$ and $[\gamma] = \Gamma_U \in Q_U \otimes U$.
	\end{defi}
	\begin{assumption}
		\label{bulk-deformation assumption}
		We assume there exists a $\YY \in Der_{\Lambda^e} Q^e_U$ be such that $[\YY(\gamma)] = c_1 \in H^*(X;Q_U^e)$.
	\end{assumption}
	\begin{defi}
		A \emph{bulk-deformation pair} over $U$ is a pair $(b,\ga)$. Here $\gamma \in \mathcal{I}_UA^*(X;Q_U)$ is a bulk-deformation parameter over $U$ and $b \in \mathcal{I}_U A^*(L,Q_U^e)$ with $|b| = 1$.
	\end{defi}
	We now recall the definitions of the bulk-deformed $\q$-operations from \cite[Section~3.4]{ST2}. For a bulk-deformation pair $(b,\ga)$ these are given by:
	\begin{equation}
		\mathfrak{q}^{b,\gamma}_{k,l}(\al_1, \dots, \al_k; \ga_1, \dots, \ga_l) := \sum_{\substack{s,t \geq 0\\s_0 + \dots + s_k = s}} \frac{1}{t!}\mathfrak{q}_{k+s,l+t}(b^{\otimes s_0} \otimes \al_1 \otimes \dots \otimes \al_k \otimes b^{\otimes s_k}; \ga_1 \otimes \dots \otimes \ga_l \otimes \gamma^{\otimes t}).
	\end{equation}
	Similarly we have:\begin{equation}
		\mathfrak{q}^{\gamma}_{\emptyset,l}(\ga_1, \dots, \ga_l) = \sum_t \frac{1}{t!}\mathfrak{q}_{\emptyset,l+t}(\ga_1 \otimes \dots \otimes \ga_l \otimes \ga^{\otimes t}).
	\end{equation}
	The bulk-deformed horocyclic $\mathfrak{q}$-operations are given by:\begin{equation}
		\mathfrak{q}^{b,\gamma}_{k,l,\perp_i}(\al_1, \dots, \al_k; \ga_1, \dots, \ga_l) := \sum_{\substack{s,t \geq 0\\s_0 + \dots + s_k = s}} \frac{1}{t!}\mathfrak{q}_{k+s,l+t,\perp_{i + \sum_j^{i-1} s_j}}(b^{\otimes s_0} \otimes \al_1 \otimes \dots \otimes \al_k \otimes b^{\otimes s_k}; \ga_1 \otimes \dots \otimes \ga_l \otimes \gamma^{\otimes t}).
	\end{equation}
	Finally, we have: \begin{equation}
		\qb_{-1,l}(\ga_1, \dots, \ga_l) := \sum_{k,t} \frac{1}{t!(k+1)} \langle \q_{k,l+t}\left(b^{\otimes k};\otimes^l_{j = 1}\ga_j \otimes \ga^{\otimes l}\right), b \rangle_L + \sum_{t} \frac{1}{t!}\q_{-1,t+l}\left(\otimes^l_{j = 1}\ga_j \otimes \ga^{\otimes l}\right).
	\end{equation}
	Similar to \cite[Lemma~4.37]{Hug}, we find:
	\begin{nlemma}
		\label{le:v of q -1}
		\begin{equation}
		v(\qb_{-1,l})(\eta) = (-1)^{|v||\eta|'} \langle \qb_{0,l}(\eta),v(b) \rangle_L + (-1)^{|v||\eta|}\qb_{-1,l+1}(\eta \otimes v(\ga)).
	\end{equation}
	\end{nlemma}
The analogue of \cite[Lemma~4.38]{Hug} is:
\begin{nlemma}
	\label{le:e of q -1}
	\begin{equation}
		e\partial_e(\qb_{-1,1})(\eta) = \langle \qb_{0,l}(\eta),e\partial_e(b) \rangle_L + \qb_{-1,l+1}(\eta \otimes c_1).
	\end{equation}
\end{nlemma}

	\begin{defi}
		\label{weak bounding pair}
		A \emph{weak bounding pair} over $U$ is a bulk-deformation pair $(b,\ga)$ over $U$ such that \begin{equation}
			\qb_{0,0} = c \cdot 1,
		\end{equation}
		for some $c \in \mathcal{I}Q_U^e$. We then say that $b$ is a \emph{bounding cochain} for $\ga$ and that $L$ is \emph{weakly unobstructed} for the bulk deformation $\ga$.
	\end{defi}
	\begin{remark}
		Establishing the existence of a bounding cochain for a general $\ga \in \mathcal{I}_UA^*(X;Q_U)$ is a difficult problem. Under simplifying cohomological assumptions, one can prove their existence. See for example \cite[Proposition~3.4]{ST1}.
	\end{remark}
	The next lemma shows that if $L$ is weakly unobstructed for some $\ga$ we can modify the bounding cochain to incorporate the deformation by $c_1$. In other words: bulk-deformations in the $c_1$ direction are unobstructed.
	\begin{nlemma}
		\label{bulk-deformation assumtion holds}
		Suppose that $(b,\ga) \in \mathcal{I}_U A^*(L,Q_U^e) \oplus \mathcal{I}_UA^*(X;Q_U)$ is a weak bounding pair over $U$. Then, there exists a vector space $U' \supset U$ with and a weak bounding pair $(b',\ga') \in \mathcal{I}_{U'} A^*(L,Q_{U'}^e) \oplus \mathcal{I}_{U'}A^*(X;Q_{U'})$ over $U'$ such that Assumption \ref{bulk-deformation assumption} holds. Furthermore, under the natural map $Q_{U'} \rightarrow Q_U$, we have that $(b',\ga') \mapsto (b,\ga)$.
	\end{nlemma}
	\begin{proof}
		Let $U' = U \oplus \CC$. We then have $Q_{U'} \cong Q_U[[t]]$ for some parameter $t$. Let $\ga' = \ga + tC_1$, for some choice of representative $C_1$ of the first Chern class of $X$, chosen such that $C_1$ also represents $\frac{\mu}{2} \in H^2(X,L)$. Let \begin{equation}
			b' = b - \sum_{i \geq 1} \frac{\left(t e\partial_e \right)^i}{i!} b.
		\end{equation} 
		First observe that by the divisor axiom \ref{divisor property}, we have: \begin{equation}
			\qb_{0,l}(C_1^{\otimes l}) = \sum_{\beta} \left(\frac{\mu(\beta)}{2}\right)^l \q^{\beta, b,\ga}_{0,0}.
		\end{equation}
		Combining this with \cite[Lemma~4.38]{Hug} with $k = l = 0$ yields: \begin{equation}
			e\partial_e \qb_{0,0} = \qb_{0,1}(C_1) - \qb_{0,0}(e\partial_e(b)).
		\end{equation}
		Repeated application of \cite[Lemma~4.38]{Hug} then allows us to compute: \begin{equation}
			(e\partial_e)^r \qb_{0,0} = \sum_{\substack{r, i_1, \dots, i_s\\ r+i_1 + \dots + i_s = k}} \binom{k}{r,i_1, \dots, i_s} (-1)^s \qb(C_1^{\otimes r}; (e\partial_e)^{i_1}b, \dots, (e\partial_e)^{i_s}b).
		\end{equation} We thus have:
		\begin{equation}
			\label{bulk-deformation computation}
			\mathfrak{q}^{b',\ga'}_{0,0} = \exp(t e\partial_e) \qb_{0,0}.
		\end{equation}
		As $\qb_{0,0} =  c \cdot 1$, this then shows that $(b', \ga')$ is a weak bounding cochain with $c' = \exp(t e\partial_e) c$. Thus Assumption \ref{bulk-deformation assumption} holds with $Y = \partial_t$.
	\end{proof}
	\subsection{Relative cyclic homology of the Fukaya $A_\infty$-algebra}
	Let $L \subset X$ be a Lagrangian submanifold satisfying assumptions \ref{assumptions}. Let $(\ga,b)$ be a weak bounding pair satisfying assumption \ref{bulk-deformation assumption}. Solomon and Tukachinsky \cite[Theorem~1]{ST3} construct an $A_\infty$-algebra $A^{ST}$ using the operations $\q^{ST}_{k,0}$. We have different sign conventions for our operations $\q$, but the following still holds.
	\begin{defi}
		\label{defi: Fukaya A infinity algebra}
		Let $(A := CF^*(L,L)[e] := A^*(L;Q^e),\mathfrak{m}_k := \qb_{k,0},\langle \;, \; \rangle_L, 1)$. This forms an $n$-dimensional, strictly unital and cyclic $A_\infty$-algebra, it is Euler-graded with Euler vector field $E = e\partial_e + E_U$. The grading operator is defined by $Gr(f\al) = (|f| + |\al|)f\al$ for $f \in Q^e$ and $\al \in A^*(L)$.
	\end{defi}
	\begin{defi}
		Let $CF^*(L,L) := CF^*(L,L)[e] \otimes_{Q^e} Q$ be the $A_\infty$-algebra obtained by setting $e=1$.
	\end{defi}
	Let $HC^-_*(A,L)$ be the relative cyclic homology associated to the $A_\infty$-algebra $A$ (viewed as an $A_\infty$-category with one object, $L$), relative to the object $L$ (see Example \ref{eg: relative cyclic chains}). \begin{nlemma}
		\label{lem: chain level connection relative Fuk}
		The connection $\nabla^L_v: CC^-_*(A,L) \rightarrow CC^-_*(A,L)$ is given on the chain level by: \begin{equation}
			\nabla^L_v(a,\al) = (v(a) - \frac{v(c)a}{u}, \nabla_v \al).
		\end{equation}
	\end{nlemma}
In the proof of the fact that the cyclic open-closed map respects the absolute connections in \cite{Hug}, we modified the connection on the cyclic homology of the Fukaya category by a homotopy $J_v$, which was given by:
\begin{defi}
	For $v \in Der_{\Lambda} Q^e$ define the operator $J_v(\al) = i\{ \phi_v \}\al$. Here the length-zero Hochschild cochain $\phi_v$ is given by $\phi_v := v(b) \in A$, and the operation $i \{ \phi_v\}$ is as defined by \cite{Ge}, or see \cite[Definition~3.12]{Hug} for a definition in our setup.
	\end{defi}
Here, we don't modify the connection on cyclic homology, but directly employ the main result obtained (\cite[Proposition~4.68]{Hug}). To show that the relative connection is respected by the relative cyclic open-closed map, we then need the following observation, which follows directly from the definition.
\begin{nlemma}
	\label{lem: Jv on unit}
	The operator $J_v$ satisfies $J_v\mathbf{e}_L = v(b)$.
\end{nlemma}
	\section{The relative quantum TE-structure}
	\label{section: the relative quantum TE-structure}
	\subsection{The T-structure}
	Solomon-Tukachinsky \cite{ST2} extend the quantum connection to a cone complex involving the Lagrangian. We will first adapt their definition to take into account the variable $u$, and then extend the connection in the $u$-direction. 
	
	Let $L \subset X$ be a Lagrangian submanifold satisfying assumptions \ref{assumptions}. Let $(\ga,b)$ be a weak bounding pair satisfying assumption \ref{bulk-deformation assumption}, with curvature $c$. First consider the chain map\begin{align*}
		\mathfrak{i}: A^*(X;\RR) &\rightarrow \RR[-n]\\
		\eta &\mapsto (-1)^{|\eta|}\int_L \eta.
	\end{align*}
	Here $\RR$ is equipped with the trivial differential. Let $C(\mathfrak{i}) = \mathbb{R}[-n-1] \oplus A^*(X)$ be the associated mapping cone, with differential: \[
	d(f,\eta) := (\mathfrak{i}(\eta),d\eta).
	\]
	\begin{defi}
		\emph{Relative quantum cohomology} is given as a vector space by: $QH^*(X,L):= H^*(C(\mathfrak{i}))$. Following \cite{ST2}, we extend the quantum connection to the relative quantum connection \begin{equation}
			\nabla^L: Der_{\Lambda} Q^e \otimes QH^*(X,L;Q_U^e)[[u]] \rightarrow u^{-1}QH^*(X,L;Q_U^e)[[u]],
		\end{equation} 
		which is given on the chain level by \begin{equation}
		\nabla_v^L(f,\eta) = \left(v(f) + u^{-1}\left(v(c)f + v(\qb_{-1,1})(\eta)\right),\nabla_v \eta\right).
		\end{equation}
		The connection $\nabla_v$ is as defined in \cite[Definition~4.4]{Hug}
	\end{defi}
	\begin{remark}
		\cite{ST2} define this connection, but without the $u$-dependence. Their connection is related to ours by setting $u = 1$, up to some sign differences.
	\end{remark}
	\begin{nlemma}[c.f.\ {\cite[Lemma~4.4]{ST2}}]
		The connection $\nabla_v^L$ descends to cohomology.
	\end{nlemma}
	\begin{proof}
		First note by keeping track of signs, it follows directly from \cite[Proposition~3.13]{ST3} and the definition of the supercommutator that $[d,\nabla_v^L] = d \circ \nabla_v^L - (-1)^{|v|} \nabla_v^L \circ d$.
		Next observe that: \begin{align}
			\label{eq:qcoh d circ nabla}
			d \circ \nabla_v^L &= \left((-1)^{|v| + |\eta|}\left(\int_L v(\eta) - u^{-1}\int_L(v(\gamma) \star \eta)\right),d\nabla_v \eta\right),\\
			\label{eq:qcoh nabla circ d}
			\nabla_v^L \circ d &= \left( u^{-1} v(\qb_{-1,1})(d\eta) + (-1)^{|\eta|}\left( v(\int_L \eta) + u^{-1}v(c) \int_L \eta\right), \nabla_v(d\eta)  \right).
		\end{align}
		Lemma \ref{boundary of q -1 operation} shows that: \begin{align}
			\qb_{-1,1}(d\eta) = (-1)^{|\eta|}\langle \qb_{0,1}(\eta),\qb_{0,0} \rangle + (-1)^{|\eta|+1}\int_L \mathfrak{q}^{\ga}_{\emptyset,1}(\eta).
		\end{align}
	Next, use the fact $\qb_{0,0} = c\cdot 1$, the top-degree property \ref{top degree property} and the energy-zero property \ref{energy zero property}, to rewrite this as:
	\begin{equation}
			\qb_{-1,1}(d\eta) = (-1)^{|\eta|+1}c \int_L \eta + (-1)^{|\eta|+1}\int_L \mathfrak{q}^{\ga}_{\emptyset,1}(\eta).
		\end{equation}
	Taking the derivative of this equation shows:\begin{equation}
			\label{v of qb -1}
			v(\qb_{-1,1})(d\eta) = (-1)^{|\eta|+1}v(c) \int_L \eta + (-1)^{|\eta|+1}\int_L v(\gamma) \star \eta.
		\end{equation}
	Substitution of equation \eqref{v of qb -1} into equations \eqref{eq:qcoh d circ nabla} and \eqref{eq:qcoh nabla circ d} shows that $[d,\nabla_{v}^L] = 0$ as required.
	\end{proof}
To compare the relative connection on relative quantum cohomology with the relative connection on relative negative cyclic homology, we need to change the signs.  We thus have:
\begin{defi}
	The sign-twisted relative connection is given by:
	\begin{equation}
		\nabla_v^{*,L}(f,\eta) = \left(v(f) - u^{-1}\left(v(c)f + v(\qb_{-1,1})(\eta)\right),\nabla^*_v \eta\right).
	\end{equation}
	The sign twisted absolute connection $\nabla^*_v$ was defined in \cite[Definition~4.10]{Hug}.
\end{defi}
\subsection{The E-structure}
\label{section: Euler-grading on relative quantum cohomology}
 Recall the Euler-grading $Gr^-: QH^*(X;Q^e)[[u]] \rightarrow QH^*(X;Q^e)[[u]]$ from \cite[Definition~4.6]{Hug}: for $\eta \in H^*(X;\RR)$ and $f \in Q_U^e[[u]]$, we have: \begin{equation}
	Gr^-(f\eta) = (|f| + |\eta| - n)f\eta = (2u\partial_u + 2E)(f)\al + 2f\mu(\al).
\end{equation}
where $\mu: H^p(X;Q_U^e) \rightarrow H^p(X;Q_U^e)$ is given by $\mu(\eta) = \frac{p-n}{2}\eta$.
\begin{defi}
	 Extend the Euler-grading to relative quantum cohomology by setting:
	\begin{equation}
		Gr_L^-(f,\eta) := ((2u\partial_u + 2E + 1)(f),Gr^-(\eta)).
	\end{equation}
\end{defi}
We will show that this grading operator makes $QH^*(X,L;Q^e)[[u]]$ into an Euler-graded T-structure.
\begin{remark}
	For an element $f \in Q^e[[u]]$ of homogeneous degree $|f|$, we have: $(2u\partial_u + 2E + 1)(f) = (|f| + 1)f$. This is the degree of $(f,0)$ in the cone complex, but shifted down by $n$. This is because the grading $Gr^-$ on quantum cohomology is also shifted down by $n$.
\end{remark}
A short computation shows:
\begin{nlemma}
	For $\lambda \in Q^e[[u]]$, the relative grading operator $Gr^-_L$ satisfies: \begin{equation}
		Gr^-_L(\lambda \cdot (f,\eta)) = (2u\partial_u + 2E)(\lambda)\cdot (f,\eta) + \lambda \cdot Gr^-_L(f,\eta). 
	\end{equation}
\end{nlemma}
We next show:
\begin{nlemma}
	The relative grading operator satisfies: \begin{equation}
		[Gr^-_L, \nabla^L_v] = \nabla^L_{[2E,v]}.
	\end{equation}
\end{nlemma}
\begin{proof}
	Observe that: \begin{equation}
		Gr^-_L \circ \nabla_v^L(f,\eta) = \left( (2u\partial_u + 2E + 1) \left( v(f) + u^{-1} \left( v(c)f + v(\qb_{-1,1})(\eta) \right) \right),Gr^- \circ \nabla_v \eta \right),
	\end{equation}
	and \begin{multline}
		\nabla^L_v \circ Gr^-_L (f,\eta) = \\ \left( v\left( (2u\partial_u + 2E + 1)(f)\right) + u^{-1} \left( v(c)\cdot(2u\partial_u + 2E + 1)(f) + v(\qb_{-1,1})(Gr^-\eta) \right),\nabla_v \circ Gr^- \eta   \right).
	\end{multline}
	We thus find: \begin{multline}
		\label{eq:euler grading commutator}
		[Gr^-_L, \nabla^L_v] =  \left( [2E,v](f), \nabla_{[2E,v]}\eta \right) \\ + \left( u^{-1} \left( (2E - 2)(v(c))f + (2u\partial_u + 2E -1)\left(v(\qb_{-1,1})(\eta)\right) - v(\qb_{-1,1})(Gr^-\eta) \right), 0 \right).
	\end{multline}
	By definition of the curvature $|c| = 2$, so that $2E(c) = 2c$. Thus, $(2E - 2)(v(c)) = [2E,v](c)$. Next, by linearity, we can assume $\eta$ is of pure cohomological degree, so that $Gr^-\eta = (deg(\eta) - n) \eta + 2E(\eta)$. Here $deg(\eta)$ denotes the cohomological degree of $\eta$. We then have: \begin{align}
		&(2u\partial_u + 2E -1)\left(v(\qb_{-1,1})(\eta)\right) - v(\qb_{-1,1})(Gr^-\eta)\\
		&= 2E\left(v(\qb_{-1,1})\right)(\eta) + v(\qb_{-1,1})\left((2u\partial_u + 2E -1)\eta \right) - v(\qb_{-1,1})((deg(\eta) - n + 2E)\eta)\\
		&= 2E\left(v(\qb_{-1,1})\right)(\eta) + v(\qb_{-1,1})((n - 1 - deg(\eta)) \eta).
	\end{align}
	Now, as $|\qb_{-1,1}(\eta)| = |\eta| - n + 1$, we find: 
	\begin{equation}
		2E(\qb_{-1,1})(\eta) = 2E(\qb_{-1,1}(\eta)) - \qb_{-1,1}(2E(\eta)) = (deg(\eta) - n +1 ) \qb_{-1,1}(\eta).
	\end{equation}
	Combining the last two equations, we thus find that \begin{equation}
		(2u\partial_u + 2E -1)\left(v(\qb_{-1,1})(\eta)\right) - v(\qb_{-1,1})(Gr^-\eta) = [2E,v](\qb_{-1,1})(\eta).
	\end{equation}
	Substituting this into equation \eqref{eq:euler grading commutator} finishes the proof.
\end{proof}
As $QH^*(X,L;Q^e)[[u]]$ is an Euler-graded T-structure, it comes with a canonical connection in the $u$-direction (see \cite[Definition~2.6]{Hug}). Writing out the formulae we have:
\begin{equation}
	\nabla^L_{\partial_u}(f,\eta) = \left(\frac{Gr^-_L}{2u} - \frac{\nabla^L_E}{u}\right)(f,\eta) = \left(\partial_u f + \frac{f}{2u} - u^{-2}\left( cf + E(\qb_{-1,1})(\eta) \right), \nabla_{\partial_u} \eta\right).
\end{equation}
\begin{remark}
	Let $(X,L)$ be a monotone symplectic manifold, with a monotone Lagrangian. We can then set the Novikov parameter $Q = 1$, also consider the case without bulk deformations.  Then, the Euler vector field is given by $E = e\partial_e$, so that $E(\qb_{-1,1})(\eta) = \q_{-1,2}(c_1, \eta)$ by Lemma \ref{le:e of q -1}. Then, setting the universal Novikov parameter $e=1$ as well, we obtain a connection $\nabla^L_{\frac{d}{du}}: QH^*(X,L;\CC)[[u]] \rightarrow QH^*(X,L;\CC)[[u]]$. It is given on the chain level by: \begin{equation}
		\nabla^L_{\frac{d}{du}}(f,\eta) = \left(\frac{df}{du} + \frac{f}{2u} - u^{-2}\left(Wf + \q_{-1,2}(c_1,\eta) \right), \nabla_{\frac{d}{du}} \eta\right).
	\end{equation}
Here $W \in \CC$ is the disk potential: the count of Maslov index 2 holomorphic disks with boundary on $L$, it appears as for monotone symplectic manifolds we have $c = W$ for degree reasons.
\end{remark}
\begin{defi}
	The sign-twisted relative connection in the $u$-direction is given by:
	\begin{equation}
		\nabla^{*,L}_{\partial_u}(f,\eta) = \left(\frac{Gr^-_L}{2u} - \frac{\nabla^{*,L}_E}{u}\right)(f,\eta) = \left(\partial_u f + \frac{f}{2u} + u^{-2}\left( E(c)f + E(\qb_{-1,1})(\eta) \right), \nabla^*_{\partial_u} \eta\right).
	\end{equation}
\end{defi}
	\section{Relative cyclic open-closed map}
	\label{section: the relative cyclic open closed map}
	\subsection{Construction}
	Let $A = CF^*(L,L)[e]$ be the Fukaya $A_\infty$ algebra defined in Definition \ref{defi: Fukaya A infinity algebra}. In this section we will define a \emph{relative cyclic open-closed map} $\mathcal{OC}^-_L: HC^{-}_*(A,L) \rightarrow QH_*(X,L;Q^e)[[u]]$. First recall the definition of the cyclic open-closed pairing $\langle \_, \mathcal{OC}^-(\_) \rangle: QH^*(X;Q^e)[[u]] \otimes HC^-_*(A) \rightarrow Q^e[[u]]$.
	\begin{defi}[{\cite[Definition~4.53]{Hug}}] The cyclic open-closed pairing is defined by:
	\begin{equation}
		\langle \eta, \OC^-(\alpha) \rangle := (-1)^{|\alpha_0|(\epsilon(\widetilde{\alpha}) + 1)}\langle \qb_{k,1}(\widetilde{\alpha};\eta),\alpha_0 \rangle_L,
	\end{equation}
	where $\widetilde{\alpha} = \alpha_1 \otimes \dots \alpha_k$ for $\alpha = \alpha_0[\alpha_1,\dots, \alpha_k]$.
\end{defi}
	\begin{remark}
		In \cite{Hug} we used the notation $\OC^-_e$ for the above pairing, and $\OC^-$ for the map obtained by evaluation at $e =1$. In the above definition we have chosen to remove the subscript $e$ to avoid cumbersome notation. As we will not use the map obtained by evaluation at $e=1$, we hope this does not lead to any confusion.
	\end{remark}We now extend this definition.
	\begin{defi}
		\label{defi: relative open-closed pairing}
		The relative open-closed pairing $\langle \_, \mathcal{OC}^-_L(\_) \rangle: C^*(\mathfrak{i}) \otimes CC_*(A,L) \rightarrow Q^e[[u]]$ is given by (the $u$-linear extension of): \begin{equation}
			\langle (f,\eta), \mathcal{OC}^-_L(a,\alpha) \rangle := \langle \eta, \mathcal{OC}^-(\alpha) \rangle + (-1)^{1+(n+1)|a|} fa.
		\end{equation}
	\end{defi}
	\begin{nlemma}
		The relative open-closed pairing descends to homology, and thus defines a pairing $$QH^*(X,L;Q^e)[[u]]\otimes HC^-_{*}(A,L) \rightarrow Q^e[[u]].$$
	\end{nlemma}
	\begin{proof}
		First we show that: \begin{equation}
			\label{eq:relative OC respects b}
			\langle d(f,\eta),\mathcal{OC}^-_L(a,\alpha) \rangle+(-1)^{|\eta|}\langle (f,\eta),\mathcal{OC}^-_L(b(a,\alpha)) \rangle = 0.
		\end{equation}
		By the top degree property (\ref{top degree property}) and the energy zero property (\ref{energy zero property}) we have: \begin{equation}
			\label{eq:open closed on unit}
			\langle \eta, \mathcal{OC}^-(ae_L) \rangle = (-1)^{|a||\eta|} a\langle \mathfrak{q}_{0,1}(\eta), e_L \rangle_L = (-1)^{|\eta| + n|a|} \mathfrak{i}(\eta)a.
		\end{equation}
		We then compute: \begin{align*}
			\langle (f,\eta),\mathcal{OC}^-_L(b(a,\alpha)) \rangle &= \langle (f,\eta),\mathcal{OC}^-_L(0,ae_L + b(\alpha)) \rangle,\\
			\intertext{which, by lemma \cite[Lemma~4.55]{Hug} and equation \eqref{eq:open closed on unit} equals:}
			&= (-1)^{|\eta|+1}\langle d\eta,\mathcal{OC}^-(\alpha) \rangle + (-1)^{|\eta| + (n+1)|a|}\mathfrak{i}(\eta)a.
		\end{align*}
		Whereas:\begin{align}
			\langle d(f,\eta), \OC^-_L(a,\alpha) \rangle &= \langle (\mathfrak{i}(\eta),d\eta), \OC^-_L(a,\al) \rangle\\
			&=\langle d\eta, \OC^-(\al) \rangle + (-1)^{1+(n+1)|a|} \mathfrak{i}(\eta)a.
		\end{align}
		So that equation \eqref{eq:relative OC respects b} holds. To finish the proof, note that: 
		\begin{equation}
		\langle (f,\eta),\mathcal{OC}^-_L(0,B\alpha) \rangle = \langle \eta, \mathcal{OC}^-(B\alpha) \rangle = 0,
		\end{equation}
		where the last equality holds by \cite[Lemma~4.60]{Hug}.
	\end{proof}
	\begin{defi}
		The relative cyclic open closed map is defined by dualising the relative open-closed pairing in the first factor to obtain a map: \begin{equation}
			\OC^-_L: HC^-_*(A,L) \rightarrow QH_*(X,L;Q^e)[[u]].
		\end{equation}
	\end{defi}
	\subsection{The relative open-closed map respects TE-structures}
	Recall that we defined $(QH_*(X,L;Q^e)[[u]], (\nabla^{*,L})^\vee)$ to be the dual TE-structure to $(QH^*(X,L;Q^e)[[u]], \nabla^{*,L})$. We will show:
	\begin{nthm}
		\label{thm: relative OC is a morphism of TE-structures}
		The relative cyclic open-closed map $$\OC^-_L: (HC^-_*(A,L),\nabla^L) \rightarrow (QH_*(X,L;Q^e)[[u]],(\nabla^{*,L})^\vee)$$ is a morphism of TE-structures.
	\end{nthm}
First recall:
\begin{defi}[{\cite[Definition~4.67]{Hug}}]
	For $v \in Der_\Lambda Q^e$, the pairing $\langle \_, G_v(\_) \rangle: A^*(L;Q^e)[[u]] \otimes CC^-_*(A) \rightarrow Q^e[[u]]$ is given by: \begin{align}
		\langle \eta, G_v(\alpha) \rangle :&= \langle \eta, \widetilde{G}_v(\alpha) \rangle + (-1)^{|v|'|\eta|'+ 1}\langle \eta, \OC^-(J_v(\alpha)) \rangle \notag\\
		&=(-1)^{|v||\eta| + |\al_0|(\epsilon(\widetilde{\al}) +1)} \langle \qb_{k,2,\perp_0}(\widetilde{\al};\eta,v(\ga)), \al_0 \rangle_L + (-1)^{|v|'|\eta|'+ 1}\langle \eta, \OC^-(J_v(\alpha)) \rangle.
	\end{align}
\end{defi}
\begin{nlemma}[{\cite[Proposition~4.68]{Hug}}]
	\label{homotopy pairing connection}
	The pairing $G_v$ satisfies:
	\begin{multline*}
	\langle \nabla^*_{v} \eta, \OC^- (\alpha) \rangle + (-1)^{|\eta||v|}\langle \eta, \OC^- (\nabla_{v} \alpha) \rangle = v \left(\langle \eta, \OC^-(\alpha)\rangle\right) \\+ u^{-1}\left(\langle d\eta, G_{v} (\alpha) \rangle + (-1)^{|\eta| + |v|}\langle \eta, G_{v}\left((b+uB)(\alpha)\right)\rangle \right).
\end{multline*}
\end{nlemma}
We extend this to the relative setting:
\begin{defi}
	The homotopy pairing\(  \langle \_, G^L_v(\_)\rangle: C^*(\mathfrak{i})[[u]] \otimes CC_*^{-}(A,L) \rightarrow Q^e[[u]] \) is defined by \[
	\langle (f,\eta),G^L_{v}(a,\alpha) \rangle = \langle \eta,G_{v}(\alpha) \rangle.
	\]
\end{defi}
We will show: \begin{nprop}
	\label{prop: relative oc respects v connection}
	The homotopy pairing satisfies:\begin{multline}
		\langle \nabla^{*,L}_v(f,\eta),\mathcal{OC}^-_L(a,\alpha) \rangle + (-1)^{|\eta||v|}\langle (f,\eta),\mathcal{OC}^-_L(\nabla^{L}_v(a,\alpha)) \rangle = v(\langle (f,\eta), \OC^-_L(a,\alpha) \rangle)\\
		+ u^{-1}(\langle d(f,\eta), G^L_v(a,\alpha) \rangle + (-1)^{|\eta|+|v|} \langle (f,\eta), G^L_v\left((b+uB)(a,\alpha)\right) \rangle)
	\end{multline}
\end{nprop}
	\begin{proof}
		We write out the terms one by one. First, we have: \begin{equation}
		\langle \nabla^{*,L}_v(f,\eta),\mathcal{OC}^-_L(a,\alpha) \rangle = \langle \nabla^*_v \eta, \OC^-(\al) \rangle + (-1)^{1+(n+1)|a|}\left(v(f) - u^{-1} \left(v(c)f + v(\qb_{-1,1})(\eta)\right)\right)a.
		\end{equation}
		Moreover, recalling the chain level formula of the connection on $HC^-_*(A,L)$ (Lemma \ref{lem: chain level connection relative Fuk}), we have:
		\begin{equation}
		\langle (f,\eta),\mathcal{OC}^-_L(\nabla^L_v(a,\alpha)) \rangle = \langle \eta, \OC^-(\nabla_v \alpha) \rangle + (-1)^{1+ (n+1)(|v|+|a|)} f\left(v(a) - u^{-1}v(c)a\right).
		\end{equation} % term +(-1)^{|a||v|}u^{-1} \langle \eta, \OC^-(av(b)) \rangle on RHS
		First observe that for an element $(f,\eta) \in C^*(\mathfrak{i};Q^e)[[u]]$ of homogeneous degree we have $|\eta| = |f| + n + 1$. We thus find:\begin{align}
			v(fa) &= v(f)a + (-1)^{(|\eta| + n+1)|v|}fv(a),\\
			fv(c) &= (-1)^{(|\eta|+ n+1)|v|}v(c)f.
		\end{align}
		Using Lemma \ref{homotopy pairing connection}, we thus have:
	\begin{multline*}
	\langle \nabla^L_v(f,\eta),\OC_L^-(a,\alpha) \rangle + (-1)^{|\eta||v|}\langle (f,\eta),\OC_L^-(\nabla^L_v(a,\alpha)) \rangle = 
	v(\langle (f,\eta), \OC_L^-(a,\alpha) \rangle) + u^{-1}\langle d\eta, G_v(\alpha) \rangle\\ + u^{-1}\left( (-1)^{|\eta|+ |v|} \langle \eta, G_v(b+uB)(\alpha) \rangle +(-1)^{(|a|+|\eta|)|v|}\langle \eta, \OC^-(av(b)) \rangle + (-1)^{(n+1)|a|} v(\qb_{-1,1})(\eta)a\right).
\end{multline*}
		Now observe that: \begin{multline}
		\langle d(f,\eta), G^L_v(a,\alpha) \rangle + (-1)^{|\eta|+|v|} \langle (f,\eta), G^L_v(b+uB)(a,\alpha) \rangle = \\ \langle d\eta, G_v(\alpha) \rangle + (-1)^{|\eta|+|v|} \langle \eta, G_v(b+uB)(\alpha) \rangle + (-1)^{|\eta|+|v| + |a|} \langle \eta, G_v(ae_L) \rangle.
		\end{multline}
		To finish the proof, we thus need to show: \begin{equation}
			\label{eq:extra terms relative OC respects conn}
			(-1)^{|\eta|+|v| + |a|} \langle \eta, G_v(ae_L) \rangle = (-1)^{(n+1)|a|} v(\qb_{-1,1})(\eta)a.
		\end{equation}
		Using Lemma \ref{unit on the horocycle 2}, we rewrite the first term of $G_v$ as:
		\begin{equation}
		\langle \eta, \widetilde{G}_v(ae_L) \rangle = (-1)^{|v||\eta| + |\eta| + |v| + n|a|} \qb_{-1,2}(\eta \otimes v(\gamma)) a.
		\end{equation}
		By Lemma \ref{lem: Jv on unit}, and the definition of $\OC^-$ we find:
		 \begin{equation}
			\langle \eta, \OC^-(J_v(e_L)) \rangle = (-1)^{(|v|+1)|a|' + n|a|} \langle \qb_{0,1}(\eta), v(b) \rangle_L a
		\end{equation}
		Thus, equation \eqref{eq:extra terms relative OC respects conn} follows from Lemma \ref{le:v of q -1}.
	\end{proof}
We have thus shown: \begin{nthm}
	\label{thm: relative cyclic open-closed is a morphism of T-structures}
	The relative cyclic open-closed map $$\OC^-_L: (HC^-_*(A,L),\nabla^L) \rightarrow (QH_*(X,L;Q^e)[[u]],(\nabla^{*,L})^\vee)$$ is a morphism of T-structures.
\end{nthm}
Next we show that the relative cyclic open-closed pairing respects the connection $\nabla^L_{e\partial_e}$. To this end, first recall: 
\begin{nlemma}[{\cite[Proposition~4.17]{Hug}}]
	\label{lem:homotopy pairing e connection}
	The pairing \begin{equation*}
		\langle \_, G_{e\partial_e}(\_) \rangle := \langle \_, G_{Y}(\_) \rangle + (-1)^{|\eta|}\langle \eta, \OC^-(J_{e\partial_e}(\alpha)) \rangle - (-1)^{|\eta|}\langle \eta, \OC^-(J_Y(\alpha)) \rangle
	\end{equation*}
 	satisfies: \begin{multline}
	\langle \nabla^*_{e\partial_e} \eta, \OC^- (\alpha) \rangle + \langle \eta, \OC^- (\nabla_{e\partial_e} \alpha) \rangle = e\partial_e \left(\langle \eta, \OC^-(\alpha)\rangle\right) \\+ u^{-1}\left(\langle d\eta, G_{e\partial_e} (\alpha) \rangle + (-1)^{|\eta|}\langle \eta, G_{e\partial_e}\left((b+uB)(\alpha)\right)\rangle \right).
\end{multline}
\end{nlemma}
We then have:
	\begin{nprop}
		\label{prop: homotopy pairing relative e connection}
		The homotopy pairing satisfies:\begin{multline}
			\langle \nabla^{*,L}_{e\partial_e}(f,\eta),\mathcal{OC}^-_L(a,\alpha) \rangle + \langle (f,\eta),\mathcal{OC}^-_L(\nabla^{L}_{e\partial_e}(a,\alpha)) \rangle = {e\partial_e}(\langle (f,\eta), \OC^-_L(a,\alpha) \rangle)\\
			+ u^{-1}(\langle d(f,\eta), G^L_{e\partial_e}(a,\alpha) \rangle + (-1)^{|\eta|} \langle (f,\eta), G^L_{e\partial_e}\left((b+uB)(a,\alpha)\right) \rangle)
		\end{multline}
	\end{nprop}
	\begin{proof}
		The proof is very similar to that of Proposition \ref{prop: relative oc respects v connection}. The only difference is that we now need to show: \begin{equation}
			\label{eq:extra terms relative OC respects connection e direction}
			(-1)^{|\eta| + |a|} \langle \eta, G_{e\partial_e}(ae_L) \rangle = (-1)^{(n+1)|a|} {e\partial_e}(\qb_{-1,1})(\eta)a. % + \langle \eta, \OC^-(a{e\partial_e}(b)) \rangle.
		\end{equation}
		To this end, first observe that as $Y(\ga) = c_1$:
	\begin{equation}
		\langle \_, G_{Y}(\_) \rangle - (-1)^{|\eta|}\langle \eta, \OC^-(J_Y(\alpha)) \rangle = (-1)^{|\eta| + n|a|} \qb_{-1,2}(\eta \otimes c_1) a.
	\end{equation}
		Equation \ref{eq:extra terms relative OC respects connection e direction} thus follows from Lemma \ref{le:e of q -1}.
	\end{proof}
To finish the proof that the relative cyclic open-closed pairing respects the connection in the $u$-direction, we need to consider the grading. Note that on $Q^e_U$ the grading is given by $ Gr^- = 2u\partial_u + 2E$. 
\begin{nlemma}
	\label{lem: relative OC respects grading}
	The relative cyclic open-closed pairing respects the grading, in the sense that: \begin{equation}
		\langle Gr^-(f,\eta), \OC^-_{L}(a,\al) \rangle + \langle (f,\eta), \OC^-_L(Gr^-(a,\al)) \rangle = Gr^- \langle (f,\eta), \OC^-_L(a,\al) \rangle.
	\end{equation}
\end{nlemma}
\begin{proof}
	This follows immediately from the definitions of the Euler gradings, and \cite[Lemma~4.70]{Hug}, which implies that: \begin{equation}
		\langle Gr^-(\eta), \OC^-(\al) \rangle + \langle \eta, \OC^-(Gr^-(\al)) \rangle = Gr^- \langle \eta ,\OC^-(\al) \rangle.
	\end{equation}
\end{proof}
As $\nabla^L_{\partial_u} = \frac{Gr^-_L}{2u} - \frac{\nabla^L_E}{u}$, Theorem \ref{thm: relative cyclic open-closed is a morphism of T-structures}, Proposition \ref{prop: homotopy pairing relative e connection} and Lemma \ref{lem: relative OC respects grading} together prove Theorem \ref{thm: relative OC is a morphism of TE-structures}.

\section{Extensions of variations of Hodge structures and open Gromov-Witten invariants}
\label{section: extensios of VHS}
\subsection{Extensions and the relative pairing}
We now consider a $\ZZ$-graded T(P)-structure $\E$ over a $\CC$-algebra $R$. As is shown in \cite{GPS} for example, such data is equivalent to a (polarised) variation of Hodge structures. We recall: 

\begin{defi}[{\cite[Lemma~2.7]{GPS}}]
	\label{defi: VHS}
	A $\CC$-variation of Hodge structures (VHS) consists of: \begin{itemize}
		\item A free, finite-rank $\ZZ/2$-graded $R$-module $\mathcal{V} \cong \mathcal{V}_{ev} \oplus \mathcal{V}_{odd}$.
		\item A flat connection $\nabla$ on each $\mathcal{V}_{\sigma}$.
		\item The decreasing Hodge filtrations \begin{equation}
			\ldots \supset \mathcal{F}^{\geq p}\mathcal{V}_{ev} \supset \mathcal{F}^{\geq p+1}\mathcal{V}_{ev} \supset \ldots
		\end{equation}
	\end{itemize}
	and 
	\begin{equation}
		\ldots \supset \mathcal{F}^{\geq p- \frac{1}{2}}\mathcal{V}_{odd} \supset \mathcal{F}^{\geq p+\frac{1}{2}}\mathcal{V}_{odd} \supset \ldots
	\end{equation}
	which satisfy Griffiths transversality: \begin{equation}
		\nabla_v \mathcal{F}^{\geq p} \subset \mathcal{F}^{\geq p-1}.
	\end{equation}
	A VHS is \emph{polarised} if it is equipped with bilinear pairings \begin{equation}
		\left( \cdot, \cdot \right): \mathcal{V}_{\sigma} \otimes \mathcal{V}_{\sigma} \rightarrow R
	\end{equation}
	for both $\sigma \in \ZZ/2$, such that \begin{equation}
		\left( \alpha, \beta \right) = (-1)^{n} \left( \beta, \al \right),
	\end{equation}
	for some $n \in \ZZ/2$, and such that \begin{equation}
		\left( \mathcal{F}^{\geq p} \mathcal{V}_{\sigma}, \mathcal{F}^{\geq q}\mathcal{V}_{\sigma} \right) = 0 \text{ if } p+q > 0,
	\end{equation}
	and such that the induced pairing \begin{equation}
		\left( \cdot, \cdot  \right): Gr^p_\mathcal{F} \mathcal{V}_{\sigma} \otimes Gr^{-p}_{\mathcal{F}}\mathcal{V}_{\sigma} \rightarrow R
	\end{equation}
	is non-degenerate, for all $p$.
\end{defi}
Suppose that we are given a short exact sequence of $\ZZ$-graded T-structures: \begin{equation}
	0 \rightarrow \E_1 \rightarrow \E_2 \rightarrow \E_3 \rightarrow 0.
\end{equation}
From the construction in \cite{GPS} it is clear that this gives rise to an exact sequence of VHS: \begin{equation}
	0 \rightarrow \mathcal{V}_1 \rightarrow \V_2 \rightarrow \V_3 \rightarrow 0.
\end{equation}
Now, suppose that $\V_1$ is a polarised VHS, and that $\V_3$ satisfies: $\F^{\geq k} \V_3 = \V_3$ for some $k \in \frac{1}{2}\ZZ$, we can then define:
\begin{defi}
	\label{defi: relative polarisation}
	The relative polarisation \begin{equation}
		\langle \cdot, \cdot \rangle: \F^{\geq -k+1}\V_1 \otimes \V_2 \rightarrow R
	\end{equation}
	is given by $\langle \cdot, \cdot \rangle := (\cdot, s(\cdot))_{\V_1}$ for some choice of splitting $s: \V_2 \rightarrow \V_1$ which respects the Hodge filtration (but not necessarily the connection). 
	\end{defi}
\begin{nlemma}
	The relative polarisation is independent of the choice of splitting $s$.
\end{nlemma}
\begin{proof}
	Given two splittings $s_1, s_2$, the difference $s_1 - s_2$ factors as: \begin{equation}
		s_1 - s_2: \V_2 \rightarrow \V_3 \xrightarrow{f} \V_1,
	\end{equation}
	for some map $f: \V_3 \rightarrow \V_1$ respecting the Hodge filtration. But then, as $\V_3 = \F^{\geq k} \V_3$, we find that $Im(s_1 - s_2) \subset \F^{\geq k} \V_1$. Then, $(\cdot, (s_1 - s_2)(\cdot))_{\V_1}$ factors through the pairing $(\cdot, \cdot)_{\V_1}: \F^{\geq -k+1}\V_1 \otimes \F^{\geq k}\V_1$, which vanishes.
\end{proof}
\subsection{A classification result}
\label{section: A classification result}
Now let $\V$ be a polarised VHS over a formal punctured disk $R = \CC((q))$. Let $\widetilde{R} = \CC[[q]]$. Let \begin{equation}
	\label{eq: extension of VHS}
	0 \rightarrow \V \xrightarrow{a} \W \xrightarrow{b} R \rightarrow 0
\end{equation}
be an extension of VHS, where $R$ is equipped with the trivial connection $\nabla_{\frac{d}{dq}} = \frac{d}{dq}$, and filtration given by $\F^{\geq k}R = R$ and $\F^{\geq k+1}R = 0$ for some fixed $k \in \frac{1}{2}\ZZ$.
\begin{defi}
	A VHS $\V$ is said to be \emph{regular singular} if there exists a $\WR$-submodule $\WV \subset \V$ with $\nabla_{q\frac{d}{dq}}\WV \subset \WV$. We say an extension of VHS as above is \emph{regular singular} if both $\V$ and $\W$ are regular singular.
\end{defi}
Regular singular VHS have a canonical $\widetilde{R}$-submodule $\widetilde{\V} \subset \V$ called the Deligne lattice (see \cite[Section~II.2.e]{Sab} for example). This is defined by the conditions that $\nabla_{q\partial_q} \widetilde{\V} \subset \widetilde{\V}$, and that the residue of the connection, $N: \WV_0 \rightarrow \WV_0$ has eigenvalues with real part in $[0,1)$. Here $\widetilde{V}_0 := \WV/q\WV$ is the fibre over $0$. Use similar notation for $\W$, with residue $M: \WW_0 \rightarrow \WW_0$. As the Deligne lattice is functorial and exact (see \cite[Proposition~5.4]{Del}), a regular singular extension of VHS gives rise to an exact sequence of $\widetilde{R}$-modules: \begin{equation}
	0 \rightarrow \WV \rightarrow \WW \rightarrow \widetilde{R} \rightarrow 0
\end{equation}
Define a filtration on the Deligne lattice by setting $\widetilde{\F}^{\geq k}\WV := \F^{\geq k}\V \cap \WV$. The following is then clear:
\begin{nlemma}
	\label{lem: right exact extension}
	Let $0 \rightarrow \V \xrightarrow{a} \W \xrightarrow{b} R \rightarrow 0$ be a regular singular extension of VHS. Then there exists a lift $f \in \widetilde{\F}^{\geq k}\WW$ with $b(f) = 1$.
\end{nlemma}
By restricting to $q=0$, and taking the residue of the connection matrices, we obtain an exact sequence of $\CC$ vector spaces, equipped with endomorphisms: 
\begin{equation}
	\label{eq:residue exact sequence}
	\begin{tikzcd}
		0 \arrow[r] &\WV_0 \arrow[loop below, "N"] \arrow[r] &\WW_0 \arrow[loop below, "M"] \arrow[r] &\CC \arrow[loop below, "0"] \arrow[r] &0.
	\end{tikzcd}
\end{equation}
\begin{nlemma}
	\label{lem: holomorphic extension and flat section}
	There exists an element $h \in \ker(\nabla^{\W}_{q\frac{d}{dq}}) \subset \WW$ such that $b(h) = 1 \in \widetilde{R}$, if and only if the residue exact sequence \eqref{eq:residue exact sequence} splits at the level of $\CC$ vector spaces with endomorphisms.
\end{nlemma}
\begin{proof}
	Assume first that \eqref{eq:residue exact sequence} splits. This means there exists an element $C \in \WW_0$, which maps to $1 \in \CC$ such that $M(C) = 0$. In block diagonal form, we can thus write the matrix for $M$ as: \begin{equation}
		\begin{pmatrix}
			N & 0\\
			0 & 0\\
		\end{pmatrix}.
	\end{equation}
	By \cite[Proposition~2.11]{Sab}, there exists a basis for $\WW$ such that $\nabla_{q\frac{d}{dq}} = q \frac{d}{dq} + M$, which shows the required element $h \in \W$ exists. The converse is clear.
\end{proof}
\begin{defi}
	\label{defi: holomorphic extenstion}
	We say a regular singular extension of VHS \begin{equation}
		0 \rightarrow \V \rightarrow \W \rightarrow R \rightarrow 0
	\end{equation}
	is \emph{holomorphically flat} if either condition in Lemma \ref{lem: holomorphic extension and flat section} holds.
\end{defi}
We now classify holomorphically flat, regular singular extensions of VHS. This result is an extension of the results in \cite{Carl}, but without taking into account the $\ZZ$-local system. See also \cite{Zuc}, who defined a generalised intermediate Jacobian for regular singular VHS.
\begin{defi}
	The generalised intermediate Jacobian is given by: \begin{equation}
		\widetilde{\mathcal{J}}^{k} = \frac{\WV}{\widetilde{\F}^{\geq k}\WV \oplus \ker(\nabla_{\frac{d}{dq}|_{\WV}})}.
	\end{equation}
\end{defi}
\begin{defi}
	A \emph{normal function} is an element $\nu \in \widetilde{\mathcal{J}}^{k}$ such that, for one (or equivalently, for all) lifts $\widetilde{\nu} \in \WV$, we have $\nabla_{q\frac{d}{dq}}\widetilde{\nu} \in \widetilde{\F}^{\geq k-1}\WV$.
\end{defi}
\begin{nlemma}
	A holomorphically flat, regular singular extensions of VHS canonically determines a normal function.
\end{nlemma}
\begin{proof}
	Consider the elements $f \in \widetilde{\F}^{\geq k}\WW$ and $h \in \ker(\nabla^{\W}_{q\frac{d}{dq}}) \subset \WW$ from Lemmas \ref{lem: right exact extension} and \ref{lem: holomorphic extension and flat section}. The difference $\widetilde{\nu} = f - h$ then satisfies $b(\widetilde{\nu}) = 0$, thus, $\widetilde{\nu} \in \WV$ is defined. Now observe that $f$ is well-defined up to an element in $a\left(\widetilde{\F}^{\geq k}\WV\right)$, and $h$ is well-defined up to an element in $a\left( \ker(\nabla_{\frac{d}{dq}|_{\WV}}) \right)$.
	By construction, $\nabla_{q\frac{d}{dq}} \widetilde{\nu} = \nabla_{q\frac{d}{dq}} f \in \widetilde{\F}^{\geq k-1}\WW$ maps to zero under $b$, and thus comes from an element in $\widetilde{\F}^{\geq k}\WV$.
\end{proof}

\begin{nlemma}
	A normal function $\nu \in \widetilde{\mathcal{J}}^{k}$ canonically determines a holomorphically flat, regular singular extensions of VHS.
\end{nlemma}
\begin{proof}
	This is very similar to \cite[Chapter~7, exercises~1 and 2]{CV2}. Choose any lift $\widetilde{\nu} \in \WV$ of $\nu$. Then define an extension by setting \begin{equation}
		\W = \V \oplus R.
	\end{equation}
	The connection is given by $\nabla^{\V} \oplus \frac{d}{dq}$ and filtration given by $\F^{\geq i}\W = span_{R} \left( (\F^{\geq i}\V,0), (\widetilde{\nu}, 1)\cdot \F^{\geq i}R \right)$, where the filtration on $R$ is defined by $\F^{\geq k}R = R$ and $\F^{\geq k+1}R = 0$. By definition of a normal function, the filtration on $\W$ satisfies Griffiths transversality.
	
	This extension is seen to be regular singular. The Deligne lattice is given by $\WW = \WV \oplus \widetilde{R}$, the element $f$ is given by $(\widetilde{\nu},1)$. It is also holomorphically flat by construction, with the element $h$ given by $(0,1)$. 
	
	The proof is then completed by showing that a different choice of lift of the normal function $\nu$ gives rise to an isomorphic extension of VHS. This is an easy exercise. 
\end{proof}
We have thus shown:
\begin{nprop}
	There exists a bijection \begin{equation}
		\lbrace \text{normal functions } \nu \in \widetilde{\mathcal{J}}^{k} \rbrace \leftrightarrow \lbrace \text{holomorphically flat, regular singular extensions of $R$ by $\V$} \rbrace 
	\end{equation}
\end{nprop}

\subsection{Open Gromov-Witten invariants from the Fukaya category}
\label{section: OGW from Fuk}
Let $X$ be a Calabi-Yau variety. Let $L \subset X$ be a Lagrangian brane with $[L] = 0 \in H_n(X)$ and with vanishing Maslov class. Assumptions \ref{assumptions} don't hold in this situation, so the results in this section are conditional on a definition of the T-structure $QH^*(X,L)[[u]]$, and a proof of Conjecture \ref{con: relative cyclic OC respects connections}, for $(X,L)$. Assuming this has been done, we will show that the Fukaya category of $X$ determines the one-point open Gromov-Witten invariants of $L$. For a definition of open Gromov-Witten invariants see for example \cite[Definition~1.5]{ST1}.

Assume $\omega \in H^{1,1}(X)$ is a rational K\"ahler class, i.e. one in $H^2(X;\QQ)$. Then, there is a natural identification $R := \Lambda \cong \CC((q))$. In this section we restrict the extension of VHS \begin{equation}
	\label{eq: quantum homology extension of VHS}
	0 \rightarrow QH_*(X) \rightarrow QH_*(X,L) \rightarrow R \rightarrow 0
\end{equation}
to the interesting parity (even degrees if $Dim_\CC X$ is odd, and odd degree if $Dim_\CC X$ is even).

Recall that the hodge filtration on $QH^*(X;R)$ is so that $F^{\geq p}QH^*(X;R) = QH^{\leq n-2p}(X;R)$. The extension \eqref{eq: quantum homology extension of VHS} is thus such that $F^{\geq \frac{1}{2}}R = R$ and $F^{\geq \frac{3}{2}}R = 0$. The relative polarisation (see Definition \ref{defi: relative polarisation}) then takes the form:
\begin{equation}
	QH_{\geq n+1}(X;R) \otimes QH_*(X,L;R) \rightarrow R. 
\end{equation}

First we show:
\begin{nlemma}
	\label{lem: projective Calabi-Yau has holomorphic T-structure}
	For a Calabi-Yau variety, the extension of VHS \eqref{eq: quantum homology extension of VHS} is holomorphically flat in the sense of Definition \ref{defi: holomorphic extenstion}.
\end{nlemma} 
\begin{proof}
	We will prove the dual statement of the lemma. Observe that the residue of the connection $$\nabla^L_{q\frac{d}{dq}}:QH^*(X,L) \rightarrow QH^*(X,L)$$ is given by $(f,\eta) \mapsto \omega \cup (f,\eta) = (0, \omega \cup \eta)$. We thus need to show that there exists a map $F: H^*(C(\mathfrak{i})) \rightarrow \CC$ such that $F(\omega \cup (f,\eta)) = 0$ for all $(f,\eta) \in H^*(C(\mathfrak{i}))$. Given $C \in H_{n+1}(X,L)$ with $\partial C = [L]$, one obtains a map \begin{align}
		F_C: H^*(C(\mathfrak{i})) &\rightarrow \CC\\
		(f,\eta) &\mapsto \int_C \eta - f.
	\end{align} 
	A short verification shows this map is indeed well defined on the relative cohomology. We then have: \begin{equation}
		F_C (\omega \cup (f,\eta)) = F_C (0,\omega \cup \eta) = \int_C \omega \cup \eta. 
	\end{equation}
	As $H^{n-1}(C(\mathfrak{i})) \cong H^{n-1}(X)$, the result then follows from the following lemma.
\end{proof}
\begin{nlemma}
	\label{lem: unique bounding cycle}
	Let $L \subset X$ be a null-homologous Lagrangian in a symplectic manifold for which the map \begin{equation*}
		[\omega] \cup: H^{n-1}(X;\CC) \rightarrow H^{n+1}(X;\CC)
	\end{equation*} 
	is an isomorphism. Then, there exists a unique $C \in H_{n+1}(X,L)$ with $\partial C = [L]$ such that \begin{equation*}
		\int_C \omega \cup \eta = 0 \text{ for all } \eta \in H^{n-1}(X).
	\end{equation*}
\end{nlemma}
\begin{proof}
	First observe that the map \begin{equation}
		\int_C \omega \cup \_ :H^{n-1}(X) \rightarrow \CC
	\end{equation} is well defined as: \begin{equation}
		\int_C \omega \cup d\eta = \int_C d(\omega \cup \eta) = \int_L \omega \cup \eta = 0.
	\end{equation}
	The last equality follows as $\omega|_{L} = 0$. 
	
	By Poincar\'e duality, we can represent the map $\int_C \omega \cup \_$ by the map $\int_X \phi_C \cup \_$ for some element $\phi_C \in H^{n+1}(X)$. By assumption, we have $\phi_C = \omega \cup \psi_C$ for some $\psi_C \in H^{n-1}(X)$. We thus have: \begin{equation}
		\int_C \omega \cup \eta = \int_X \omega \cup \psi_C \cup \eta \text{ for all } \eta \in H^{n-1}(X).
	\end{equation}
	This defines a map \begin{align}
		H_{n+1}(X,L) &\rightarrow H^{n-1}(X)\\
		C &\mapsto \psi_C.
	\end{align}
	The composition $H_{n+1}(X) \rightarrow H_{n+1}(X,L) \rightarrow H^{n-1}(X)$ is the Poincar\'e duality isomorphism. The map $C \mapsto \psi_C$ thus has one-dimensional kernel. The required $C$ is the unique generator of the kernel with $\partial C = [L]$.
\end{proof}
\begin{remark}
	The cohomological condition is a weakening of the hard Lefschetz property, and thus holds for any K\"ahler manifold. In particular for Calabi-Yau varieties.
\end{remark}

We now construct a useful model for $QH_*(X,L;R)$. To this end, consider the map \begin{align*}
	\mathfrak{j}: \RR  &\rightarrow C_*(X;\RR)\\
	1 &\mapsto [L]
\end{align*}
Let $Cone(\mathfrak{j}) := \RR \oplus C_*(X;\RR)$ be the cone over $\mathfrak{j}$. As a vector space we can represent $QH_*(X,L)$ by the homology $H_*(Cone(\mathfrak{j})) \otimes_{\RR} R$ as follows. The pairing \begin{align*}
	Cone(\mathfrak{j}) \otimes C(\mathfrak{i}) \rightarrow R\\
	(a,C) \otimes (f,\eta) \mapsto  af - \int_C \eta
\end{align*}
descends to homology, and is seen to be non-degenerate. Thus, we get an isomorphism of vector spaces $QH_*(X,L) \cong H_*(Cone(\mathfrak{j})) \otimes R$. 

Ganatra-Perutz-Sheridan show:

\begin{nlemma}[{\cite[Lemma~3.4]{GPS}}]
	\label{lem: GPS lemma}
	The data of the VHS $QH^*(X;\Lambda)$ determines the degree grading on $QH^*(X;R)$, the constants $H^*(X;\CC) \subset QH^*(X;R)$, the holomorphic volume form $\Omega = 1 \in QH^*(X;R)$ and the matrix $A(q) \in End(QH^*(X;\Lambda))$ for the quantum multiplication $\omega \star$.
\end{nlemma}

By Lemmas \ref{lem: projective Calabi-Yau has holomorphic T-structure} and \ref{lem: holomorphic extension and flat section} we obtain an element $h \in QH_*(X,L)$ with $\nabla^{L,\vee}(h) = 0$, and $h \mapsto 1 \in \WR$. Note that $h$ is well-defined up to addition of elements in $Ker(\nabla^\vee_{\frac{d}{dq}}): QH_*(X;\WR) \rightarrow QH_*(X;\WR)$. Such elements are flat sections without monodromy. \begin{nlemma}
	Let $g \in Ker(\nabla_{\frac{d}{dq}}): QH^*(X;\WR) \rightarrow QH^*(X;\WR)$. Let $\eta \in H^*(X;\CC)$ be the lowest degree term appearing in $g$ (which is necessarily constant by the flatness condition). Then $\eta \in H^{\geq n+1}(X;\CC)$.
\end{nlemma}
\begin{proof}
	Suppose $\eta \in H^{p}(X;\CC)$. Let $\zeta \in H^{p+2}(X;\WR)$ be degree $p+2$ part of $g$. Flatness of $g$ implies that \begin{equation}
		\omega \star \eta = Q \frac{d \zeta}{dQ}.
	\end{equation}
	Thus, if $\omega \star \eta$ has a constant term, this has no solutions for $\zeta \in H^{p+2}(X;\WR)$. Hence, the constant term, which is $\omega \cup \eta$, must vanish. We then apply the hard Lefschetz theorem to find that any class $\eta$ with $\omega \cup \eta = 0$ must be of degree at least $n$. But, we have restricted to the cohomology of parity opposite to the dimension, so the lowest degree is at least $n+1$. 
\end{proof}
Thus, for degree reasons, any choice of $h$ must be represented by \begin{equation}(1,C) + \sum_{i \leq n-1} (0,f_i) \in H_*(Cone(\mathfrak{j}))\otimes \widetilde{R} \text{ with } f_i \in H_{i}(X;\WR),
\end{equation}
where $C$ is the cycle bounding $L$ constructed in Lemma \ref{lem: unique bounding cycle}. The previous lemma shows that $f_{n-1}$ is independent of the choice of $h$, up to a constant element, that is, one in $H_{n-1}(X;\CC)$. By Lemma \ref{lem: GPS lemma}, the VHS $QH_*(X)$ determines the constant elements $QH_{n+1}(X;\CC)$. We have thus shown: \begin{nlemma}
	The extension of VHS \eqref{eq: quantum homology extension of VHS} determines the map\begin{equation}
		F = \langle \_, (0,f_{n-1}) \rangle_{QH_*(X,L)}: QH_{n+1}(X;\CC) \rightarrow \WR/\CC.
	\end{equation}
\end{nlemma} 
We will now show this map is exactly the open Gromov-Witten invariant with 0 boundary marked points and 1 interior marked point.

\begin{nlemma}
	The map $F \circ PD: H^{n-1}(X;\CC) \rightarrow \WR/\CC$ is the map \begin{equation}
		OGW_{\star,0}(\_) = \sum_{d \in H_2(X,L)} q^{\omega(d)} OGW_{d,0}(\_): H^{n-1}(X;\CC) \rightarrow \WR/\CC
	\end{equation}
\end{nlemma}
\begin{proof}
	Let $\eta_1, \dots, \eta_N \in H_{n-1}(X;\CC)$ be a basis. Let $\eta^1, \dots, \eta^N \in H^{n-1}(X;\CC)$ be the dual basis. Write $f_{n-1} = \sum^N_{i=1} g_i\eta_i$ for some functions $g_i \in \WR$, which are well-defined up to a constant. The condition $\nabla^{L,\vee}(h) = 0$ states that for any $(f,\eta) \in QH^*(X,L)$, we have \begin{equation}
		\label{eq: condition on h 3}
		h(\nabla^L_{q\frac{d}{dq}}(f,\eta)) = q\frac{d}{dq}h(f,\eta). 
	\end{equation}
	Apply this equation to $(f,\eta) = (0, \eta^i)$ to find: \begin{equation}
		q\frac{dg_i}{dq} = \q_{-1,2}(\omega, \eta^i) - \int_C \q_{\emptyset,2}(\omega, \eta^i) = OGW_{*,0}(\omega,\eta^i),
	\end{equation}
	where the latter equality holds by definition of $OGW$ in this situation. Thus, applying the divisor axiom, we find that $g_i = OGW_{*,0}(\eta^i)$. For $PD(\eta^i) \in H_{n+1}(X;\CC)$, we thus have: \begin{equation}
		\langle PD(\eta^i), h \rangle_{QH_*(X,L)} = OGW_{*,0}(\eta^i)
	\end{equation}
	as required.
\end{proof}
As the map $F$ is determined by the extension of VHS, it can be obtained from the Fukaya category. We have thus proved Theorem \ref{thm: OGW from Fuk quintic threefold}.
\begin{remark}
	In \cite{ST1}, open Gromov-Witten invariants take the form $OGW_{d,k}: H^*(C(\mathfrak{i}))^{\otimes l} \rightarrow \Lambda$. In our setup, the invariants we consider take inputs in degree $n-1$, for which there is a canonical isomorphism $H^{n-1}(C(\mathfrak{i})) \cong H^{n-1}(X)$. 
\end{remark}

%\begin{nthm}
%	Let $X$ be the Fermat quintic threefold, and $L = L^+ - L^-$ be the real locus equipped the two local systems. Then, assuming conjecture \ref{con: relative cyclic OC respects connections} has been proven for $(X,L)$, the Fukaya %category of $X$ determines the open Gromov-Witten invariants of $(X,L)$. 
%\end{nthm}
\appendix

	\printbibliography
\end{document}